\documentclass[11pt]{article}

\usepackage[a4paper,margin=1in]{geometry}
\usepackage[T1]{fontenc}
\usepackage[utf8]{inputenc}
\usepackage{lmodern}
\usepackage{microtype}
\usepackage{amsmath,amssymb,amsthm,mathtools,mathrsfs}
\usepackage{graphicx}
\usepackage{tikz}
\usetikzlibrary{arrows.meta}
\usepackage{enumitem}
\usepackage{needspace}
\usepackage{hyperref}

\graphicspath{{figures/}}

\hypersetup{
  colorlinks=true,
  linkcolor=blue,
  citecolor=blue,
  urlcolor=blue
}

\newtheorem{theorem}{Theorem}[section]
\newtheorem{proposition}[theorem]{Proposition}
\newtheorem{lemma}[theorem]{Lemma}
\newtheorem{corollary}[theorem]{Corollary}
\theoremstyle{definition}
\newtheorem{definition}[theorem]{Definition}
\theoremstyle{remark}

\DeclareMathOperator{\Interleave}{Interleave}
\DeclareMathOperator{\Rec}{Rec}
\DeclareMathOperator{\FP}{FP}
\DeclareMathOperator{\Img}{Im}
\DeclareMathOperator{\Comp}{Comp}
\DeclareMathOperator{\rev}{rev}

\newcommand{\Qt}{\widetilde Q}
\newcommand{\N}{\mathbb N}
\newcommand{\Star}{*}
\newcommand{\Lam}{\Lambda}
\newcommand{\multiset}[1]{\{\!\{#1\}\!\}}

\title{The Mantovanelli-Hofstadter Sequence}
\author{Beno\^it Cloitre\\
\small Independent researcher, Paris, France\\
\small Corresponding author: \href{mailto:benoit.cloitre@proton.me}{benoit.cloitre@proton.me}\\
\small ORCID: \href{https://orcid.org/0009-0001-6778-153X}{0009-0001-6778-153X}}
\date{}

\begin{document}

\maketitle

\begin{abstract}
We study the perturbed Hofstadter recurrence introduced by Mantovanelli,
\[
 \Qt(n)=\Qt(n-\Qt(n-1))+\Qt(n-\Qt(n-2))+(-1)^n,
 \qquad \Qt(1)=\Qt(2)=1.
\]
We prove that this recurrence is well-defined for every positive integer and
that
\[
 \lim_{n\to\infty}\frac{\Qt(n)}n=\frac12.
\]
We establish the optimal order
\[
 \Qt(n)=\frac n2+O\!\left(\frac{n}{\sqrt{\log n}}\right),
\]
and give explicit positive lower and upper bounds for the corresponding
normalized limsup.  The proof is primarily combinatorial.  An odd-even split
turns the recurrence into two exact interleavings of binary words, and the
resulting Dyck paths define plane forests.  Catalan numbers also arise in the
enumeration.
\end{abstract}



\section{Introduction}

Hofstadter's \(Q\)-sequence is OEIS A005185 \cite{Hofstadter79,OEIS}.  It is
the best-known nested recurrence
\[
 Q(n)=Q(n-Q(n-1))+Q(n-Q(n-2)),
 \qquad Q(1)=Q(2)=1.
\]
It is not known whether this recurrence defines a value at every positive
integer, nor whether \(Q(n)/n\) has a limit.  Deane and Gentile studied a
diluted family obtained by retaining one nested term of Hofstadter's
recurrence and replacing the other by a prescribed slow sequence.  They
proved well-definedness for this family \cite{DeaneGentile25}.  Related
meta-Fibonacci recurrences and variants of \(Q\) are studied in
\cite{AFA17,Alkan18,BKT07,Stoll09,Fox17}.
Among the early tractable relatives of \(Q\), Tanny's well-behaved cousin is
monotone and attains every positive integer \cite{Tanny92}.
Ruskey showed that suitable initial conditions can make the same recurrence
produce values of the Fibonacci sequence on one residue class
\cite{Ruskey11}.

For comparison, Conway's \$10,000 sequence is OEIS A004001 \cite{OEIS}.  It
is defined by
\[
 C(1)=C(2)=1,
 \qquad
 C(n)=C(C(n-1))+C(n-C(n-1)).
\]
Its ratio tends to \(1/2\), and its convergence and recursive structure were
analyzed by Mallows and by Kubo and Vakil \cite{Mallows91,KuboVakil96}.
Grytczuk studied a triply nested variant \cite{Grytczuk04}.  The corresponding
conclusions remain open for \(Q\).

Mantovanelli introduced the perturbed sequence \(\Qt\), OEIS A394051
\cite{Mantovanelli26,OEIS}, defined by
\begin{equation}
\label{eq:Qt}
 \Qt(n)=\Qt(n-\Qt(n-1))+\Qt(n-\Qt(n-2))+(-1)^n,
 \qquad \Qt(1)=\Qt(2)=1.
\end{equation}
Mantovanelli's initial numerical study revealed the dyadic self-similarity of
the sequence \cite{Mantovanelli26}.  A subsequent preprint established
well-definedness by certified finite-state induction
\cite{MantovanelliFSI}.  The present paper gives a direct self-contained proof
without a finite-state certificate.  It describes the words and forests that
generate the arches and determines their amplitudes and asymptotic
consequences.

Figure~\ref{fig:comparison} contrasts the original and perturbed recurrences.
Exact identities on words and forests determine the arch structure of the
perturbed sequence.

\begin{figure}[htbp]
\centering
\includegraphics[width=\textwidth]{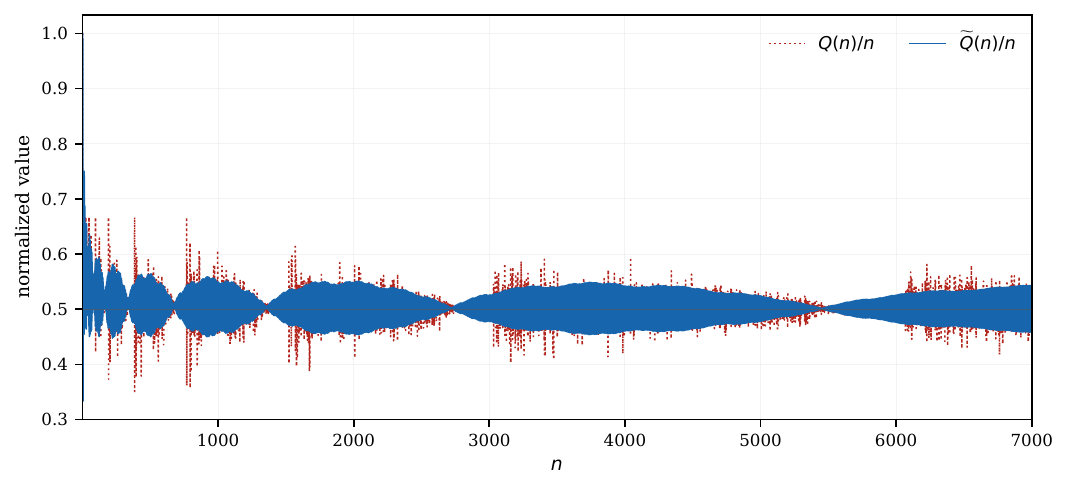}
\caption{Normalized values of Hofstadter's sequence \(Q(n)/n\) and of the
perturbed sequence \(\widetilde Q(n)/n\), for \(1\leq n\leq7000\).}
\label{fig:comparison}
\end{figure}

All values of \(\Qt\) are odd.  Splitting odd and even positions gives two
slowly increasing sequences
\[
 \mathfrak a(m)=\frac{\Qt(2m-1)+1}{2},
 \qquad
 \mathfrak b(m)=\frac{\Qt(2m)+1}{2}.
\]
Their difference \(\delta(m)=\mathfrak b(m)-\mathfrak a(m)\) is a walk made of alternating
positive and negative arches.  The increments of \(\mathfrak a\) and \(\mathfrak b\) are binary,
so each positive arch is encoded by a word \(P_r\).  For a binary word \(w\),
the associated lattice path has height
\[
 h_w(t)=Z_w(t)-O_w(t),
\]
where \(Z_w(t)\) and \(O_w(t)\) count zeros and ones in the prefix of length
\(t\).  Put \(Z(w)=Z_w(|w|)\) and \(O(w)=O_w(|w|)\).  Thus
\[
 V_r=\max_{0\leq t\leq |P_r|}h_{P_r}(t)
\]
is the maximum height reached by the walk of the \(r\)-th positive arch, not
a maximum of the letters of the word.

We use the standard encoding of a Dyck word by a plane forest.  For the core
\(A_r=P_r[1:|P_r|-1)\), the \(m\)-th zero represents vertex \(m\).  The
number \(c_m\) of ones since the preceding zero is its degree, and
\(\omega_m=c_m+1\) is the corresponding zero-gap label.  The orbital cores
have two additional properties.  Each is
fixed by reversal followed by the exchange of zeros and ones.  Moreover, at a
vertex of degree \(L\), the labels of its children occur in a prescribed order
and form a permutation of \(1,\ldots,L\).  This local rule is denoted by
\(\mathsf S\) in Section~\ref{sec:gaps-symmetry}.

Two exact-fit interleaving laws generate every arch and prove constructively
that \eqref{eq:Qt} never requests an undefined value.  On the associated
forests, the symmetry and the local gap rule force any two nonterminal edges
of the same span to carry the same label.  This is the Toeplitz property used
below.  It synchronizes the cuts at which a new path height is first attained
and allows these cuts to be transported through the word transform.  A
compression map then groups record heights with the same image into ordered
fibers.  The fiber sizes determine the component sizes at the next level.
Thus a component of a given size produces a prescribed list of new component
sizes.  These are the succession rules of the Enumerating Combinatorial
Objects (ECO) method.

The binary-word notation is standard in combinatorics on words
\cite{Lothaire02}, and we use the usual correspondence between Dyck paths and
plane forests.  Related morphic descriptions of Hofstadter-type functions
appear in \cite{LLS26,Shallit25}.  The ECO method enumerates
recursively generated classes from succession rules \cite{Barcucci99}.  Here
the local gap rule, the Toeplitz condition on forest edges, and the
record-to-fiber correspondence are specific to the present sequence.  The
resulting generating function counts full binary trees with a marked vertex
and yields the exact binomial arch amplitudes.  Similar labeled-tree encodings
occur for nested recurrences \cite{JR05,Fox22} and for constructions of
non-monotonic solutions \cite{SunoharaTannyVerberne19}.
Generating functions and asymptotic estimates are used only after the word
and forest identities have been proved \cite{FlajoletSedgewick09}.

The amplitude law determines the optimal order of the error.  Binomial layers
below the first positive summit yield a family of explicit lower bounds for
the normalized limsup, while domination of the negative arches gives an upper
bound uniform over all indices.  Determining the exact limsup constant remains
open.

All words in this article are indexed from zero.  A factor \(w[a:b)\) contains
the letters with indices \(a,a+1,\ldots,b-1\).  Heights are evaluated at
cuts.  We set \(\N=\{0,1,2,\ldots\}\).  The notation \(\log\) means the
natural logarithm, while \(\log_2\) means the base-two logarithm.

\subsection{Main theorem}

Section~\ref{sec:construction} gives positive arch words \(P_r\), their
starting positions \(u_r\), and negative arch words \(N_r\).  We write
\[
 a_r=\frac{|P_r|}{2},
 \qquad v_r=u_r+2a_r,
 \qquad
 V_r=\max_{0\leq t\leq |P_r|}h_{P_r}(t),
 \qquad
 W_r=V_r-V_{r-1}\quad(r\geq1).
\]
\Needspace{5\baselineskip}
\begin{theorem}
\begin{enumerate}[label=\textup{(\roman*)}]
\item The recurrence \eqref{eq:Qt} defines a unique integer \(\Qt(n)\) for
every \(n\geq1\), and
\[
 \lim_{n\to\infty}\frac{\Qt(n)}n=\frac12.
\]

\item For every \(r\geq1\),
\[
 W_r=\binom{2r+1}{r}.
\]
This is the term of index \(r\) of OEIS A001700 \cite{OEIS}.
Consequently,
\[
 V_r=2+\sum_{j=1}^r\binom{2j+1}{j}
 \sim\frac{8}{3\sqrt\pi}\frac{4^r}{\sqrt r}.
\]
Here \(V_r\) is the term of index \(r+1\) of OEIS A024718 \cite{OEIS}.

\item One has
\[
 \Qt(n)=\frac n2+O\!\left(\frac{n}{\sqrt{\log n}}\right).
\]
The error is not \(o(n/\sqrt{\log n})\).

\item The explicit constant \(C_{\mathrm{layer}}\) defined in
\eqref{eq:layer-constant} satisfies
\[
 C_{\mathrm{layer}}
 \leq\limsup_{n\to\infty}
 \left(\frac{\Qt(n)}n-\frac12\right)\sqrt{\log_2 n}
 \leq\limsup_{n\to\infty}
 \left|\frac{\Qt(n)}n-\frac12\right|\sqrt{\log_2 n}
 \leq\frac1{2\sqrt{2\pi}}.
\]
Moreover,
\[
 C_{\mathrm{layer}}>\frac1{3\sqrt{2\pi}}.
\]
For each fixed \(z>0\), Theorem~\ref{thm:layer-subsequences} gives the
corresponding explicit even subsequence and its exact limit.
\end{enumerate}
\end{theorem}

The formula defining the lower constant, its Gaussian subsequences, and the
numerical orientation are postponed to Section~\ref{sec:asymptotics}, where they are used.  The
strict lower inequality is exact, not numerical.  Whether the full limsup
equals \(C_{\mathrm{layer}}\) is open.

\subsection{Proof outline}

The parity split gives two exact interleaving laws.  They construct every arch
and establish well-definedness.  Each positive core is fixed by
reversal-complementation and satisfies the local gap rule.  Its gap data
define a plane forest, and the possible root words form three parametric
families.

Reflection and the gap rule give the Toeplitz property.  The resulting labels
define a potential whose minima locate the record heights.  Record transport
under \(T\) and \(T^2\) preserves the positions needed to follow the maximum
\(V_r\).  Compression gives ordered fibers.  Their sizes agree with the
Toeplitz labels, and adjacency supplies the component recurrence.  The
resulting ECO rules yield
\(W_r=\binom{2r+1}{r}\) through a generating function.

The binomial amplitude gives the optimal order of the error.  It bounds the
positive arches, and the domination theorem gives the same estimate on the
negative arches.  The forest layers and their Gaussian limits give the
explicit subsequences and the limsup bounds.

\subsection{A running example}

The smallest orbital arch displays the passage from a word to its Dyck path,
canonical forest, record cuts, and first fiber.  This section is a preview.
The zero coordinates are defined in Subsection~\ref{sec:notation}, the transformation \(T\)
in Section~\ref{sec:transform-T}, and the compression map and fiber notation
in Section~\ref{sec:fiber-map}.  Start from
\[
 P_0=001011,
 \qquad A_0=P_0[1:5)=0101.
\]
The height vector of \(P_0\), at its seven cuts, is
\[
 (0,1,2,1,2,1,0).
\]
The word \(P_0\) is a primitive Dyck word and its first record cuts are
\(\rho=(0,1,2)\), with \(V_0=2\).  For the core \(A_0\), the zero cuts and
their derived data are
\[
 p=(0,1,3),\qquad \omega=(1,2),\qquad c=(0,1),
 \qquad q=(0,0,1),\qquad d=(0,1,1).
\]
The canonical forest has root \(2\), of degree one, and its only
child is the leaf \(1\).  Its root word is \(R=(1)\).  The record cuts of
\(A_0\) are \(0\) and \(1\).

In the notation of Section~\ref{sec:fiber-map}, the compression of \(P_0\) sends
all three first-passage parameters to the first zero cut:
\[
 \bigl(\psi_{P_0}(\rho_{P_0,h})\bigr)_{h=0}^2=(1,1,1),
 \qquad g_{P_0}=(1,1,1).
\]
The resulting fiber data are
\[
 I(P_0)=\{1\},\qquad
 \operatorname{Fib}_{P_0}(1)=\{0,1,2\},\qquad f_0=(3).
\]
The two transformations are also explicit:
\[
 T(A_0)=0010011011=0N_0
\]
and
\[
 T^2(A_0)=0001000101100101110111=P_1.
\]
For comparison, the first record cuts of \(P_1\) are
\[
 (0,1,2,3,6,7),
\]
and its compression values are \((1,1,1,1,2,2)\).  Its fibers are
\[
 \operatorname{Fib}_{P_1}(1)=\{0,1,2,3\},\qquad
 \operatorname{Fib}_{P_1}(2)=\{4,5\},\qquad f_1=(4,2).
\]
These finite identities follow by direct calculation.  They preview the
definitions below and are not used in the proof of any general statement.
Figure~\ref{fig:running-example} summarizes these data.

\begin{figure}[htbp]
\centering
\begin{tikzpicture}[>=Stealth,scale=0.82,transform shape,
  every node/.style={font=\small}]
  \tikzset{panel/.style={rounded corners,draw=black!55,minimum width=3cm,
    minimum height=3.3cm}}

  \node[panel] at (0,0) {};
  \node[font=\bfseries] at (0,1.25) {word};
  \node[align=center] at (0,0.25) {$P_0=001011$\\[-1mm]
    $A_0=0101$};
  \node[font=\scriptsize,align=center] at (0,-1.05)
    {delete the first\\and last letters};

  \begin{scope}[xshift=3.7cm]
    \node[panel] at (0,0) {};
    \node[font=\bfseries] at (0,1.25) {Dyck path};
    \draw[->,black!45] (-1.25,-0.9)--(1.35,-0.9);
    \draw[->,black!45] (-1.25,-0.95)--(-1.25,0.9);
    \draw[very thick,blue!65!black]
      (-1.2,-0.9)--(-0.8,-0.2)--(-0.4,0.5)--(0,-0.2)--
      (0.4,0.5)--(0.8,-0.2)--(1.2,-0.9);
    \node[font=\scriptsize] at (0,-1.25) {$h=(0,1,2,1,2,1,0)$};
  \end{scope}

  \begin{scope}[xshift=7.4cm]
    \node[panel] at (0,0) {};
    \node[font=\bfseries] at (0,1.25) {forest of $A_0$};
    \node[draw,circle,inner sep=2pt] (r) at (0,0.45) {$2$};
    \node[draw,circle,inner sep=2pt] (l) at (0,-0.65) {$1$};
    \draw[->] (r)--(l);
    \node[font=\scriptsize] at (0,-1.25) {$c=(0,1)$, $R=(1)$};
  \end{scope}

  \begin{scope}[xshift=11.1cm]
    \node[panel] at (0,0) {};
    \node[font=\bfseries] at (0,1.25) {record cuts};
    \node[align=center] at (0,0.25) {$\rho=(0,1,2)$\\$V_0=2$};
    \node[font=\scriptsize,align=center] at (0,-1.05)
      {$V_0$ is the final\\first-passage height};
  \end{scope}

  \begin{scope}[xshift=14.8cm]
    \node[panel] at (0,0) {};
    \node[font=\bfseries] at (0,1.25) {fiber};
    \node (h0) at (-0.75,0.55) {$0$};
    \node (h1) at (-0.75,0) {$1$};
    \node (h2) at (-0.75,-0.55) {$2$};
    \node[draw,circle,inner sep=2pt] (u) at (0.75,0) {$1$};
    \draw[->] (h0)--(u); \draw[->] (h1)--(u); \draw[->] (h2)--(u);
    \node[font=\scriptsize] at (0,-1.25) {$f_0=(3)$};
  \end{scope}

  \draw[->,thick] (1.55,0)--(2.15,0);
  \draw[->,thick] (5.25,0)--(5.85,0);
  \draw[->,thick] (8.95,0)--(9.55,0);
  \draw[->,thick] (12.65,0)--(13.25,0);
\end{tikzpicture}
\caption{The running example \(P_0=001011\), from its core word to its first
fiber.}
\label{fig:running-example}
\end{figure}

\section{Arch construction}
\label{sec:construction}

\subsection{Notation}
\label{sec:notation}

The zero coordinates of a Dyck word are used in the word and forest
constructions.
\begin{definition}
\label{def:zero-coordinates}
Let \(W\) be a Dyck word of length \(2n\).  Its zero cuts are
\[
 0=p_0<p_1<\cdots<p_n,
\]
where \(p_m\) is the cut immediately after the \(m\)-th zero.  For
\(1\leq m\leq n\), set
\[
 \omega_m=p_m-p_{m-1},
 \qquad c_m=\omega_m-1.
\]
Thus \(\omega_m\) is the \(m\)-th zero-gap length, and \(c_m\) is the number
of intervening ones.  Later, \(c_m\) is the degree of vertex \(m\).  For
\(0\leq m\leq n\), set
\[
 q_m=p_m-m,
 \qquad d_m=m-q_m.
\]
The number \(q_m\) counts the ones before the \(m\)-th zero, and
\(d_m=h_W(p_m)\).  In particular,
\[
 p_0=q_0=d_0=0.
\]
Let \(s_m\), \(1\leq m\leq n\), be the cut immediately after the \(m\)-th
one, and put \(s_0=0\).  On the extended domain
\(\{1,\ldots,2n+1\}\), the terminal sentinels are
\[
 p_{n+1}=s_{n+1}=2n+1.
\]
They mark the end of the domain.  They are not zero or one cuts, and neither
\(\omega_{n+1}\) nor \(c_{n+1}\) is defined.
\end{definition}
For \(0\leq t\leq|W|\), put
\begin{equation*}
 \theta_W(t)=p_{O_W(t)},
 \qquad e_W(t)=t-\theta_W(t).
\end{equation*}

\subsection{Parity split}

Whenever the values involved are already defined, they are odd.  Put
\[
 \widehat Q(n)=\frac{\Qt(n)+1}{2},
 \qquad \mathfrak a(m)=\widehat Q(2m-1),
 \qquad \mathfrak b(m)=\widehat Q(2m).
\]

\Needspace{5\baselineskip}
The recurrence then splits along the two parities.
\begin{lemma}
\label{lem:parity-split}
Let \(m\geq2\).  Assume that the values of \(\Qt\) needed below are defined,
odd, and evaluated at positive arguments.  Then
\begin{align}
\label{eq:RA}
 \mathfrak a(m)&=\mathfrak b(m-\mathfrak b(m-1))+\mathfrak b(m-\mathfrak a(m-1))-1,\tag{RA}\\
\label{eq:RB}
 \mathfrak b(m)&=\mathfrak a(m+1-\mathfrak a(m))+\mathfrak a(m+1-\mathfrak b(m-1)).\tag{RB}
\end{align}
Conversely, these two recurrences recover \(\Qt\) through
\[
 \Qt(2m-1)=2\mathfrak a(m)-1,
 \qquad \Qt(2m)=2\mathfrak b(m)-1.
\]
\end{lemma}

\begin{proof}
An induction gives odd values whenever the recurrence is defined.  At the
odd index \(2m-1\), its two arguments are
\[
 2(m-\mathfrak b(m-1)),\qquad 2(m-\mathfrak a(m-1)).
\]
Substitution of \(\Qt(2k)=2\mathfrak b(k)-1\) gives \eqref{eq:RA}.  At the even index
\(2m\), the two arguments are
\[
 2(m+1-\mathfrak a(m))-1,
 \qquad 2(m+1-\mathfrak b(m-1))-1.
\]
Substitution of \(\Qt(2k-1)=2\mathfrak a(k)-1\) gives \eqref{eq:RB}.
\end{proof}

We construct \(\mathfrak a\) and \(\mathfrak b\) without assuming that the original recurrence
continues.  Lemma~\ref{lem:parity-split} is applied only after the
construction has been completed.

\subsection{Interleaving and arch words}

\Needspace{5\baselineskip}
The basic operation is the following two-tape interleaving.
\begin{definition}
Let \(X\) and \(Y\) be finite binary words and let
\(\varepsilon\in\{0,1\}\).  The word
\(\Interleave(X,Y,\allowbreak\varepsilon)\) is emitted by a two-tape transducer.  In
state zero it reads the next letter of \(X\).  In state one it reads the next
letter of \(Y\).  The emitted letter becomes the next state.  If the
prescribed tape is empty and the other tape is not, the machine must read
from the other tape.  It stops exactly when both tapes are empty.  An
interleaving has \emph{exact fit} if this fallback never occurs and both tapes
are exhausted.
\end{definition}

For \(Z=\Interleave(X,Y,\varepsilon)\), let \(i_t,j_t\) be the numbers of
letters consumed from \(X,Y\) after \(t\) outputs.

\Needspace{5\baselineskip}
Heights are additive along the two tapes.
\begin{lemma}

For every cut \(t\),
\[
 h_Z(t)=h_X(i_t)+h_Y(j_t).
\]
\end{lemma}

\begin{proof}
The prefix \(Z[0:t)\) is a shuffle of \(X[0:i_t)\) and \(Y[0:j_t)\).
Zero counts and one counts are additive under a shuffle.
\end{proof}

\Needspace{5\baselineskip}
Before any fallback, the head positions are determined by the emitted prefix.
\begin{lemma}
\label{lem:head-counts}
Suppose no fallback occurs before output \(Z_t\), where \(t\geq1\).  Before
that output, the two heads are
\[
 i_t=\mathbf1_{\{\varepsilon=0\}}+Z_Z(t-1),
 \qquad
 j_t=\mathbf1_{\{\varepsilon=1\}}+O_Z(t-1).
\]
Here \(Z_Z(t-1)\) and \(O_Z(t-1)\) count the letters among
\(Z_0,\ldots,Z_{t-2}\).
\end{lemma}

\begin{proof}
The first output is read from the tape prescribed by \(\varepsilon\).  Every
later output is read from \(X\) or \(Y\) according as the preceding output is
zero or one.  Counting those reads gives the two formulas.
\end{proof}

Define
\begin{equation*}
 a_0=3,
 \qquad a_{r+1}=4a_r-1,
 \qquad u_r=2a_r-r-2,
 \qquad v_r=4a_r-r-2.
\end{equation*}
Thus
\[
 a_r=\frac{2\cdot4^{r+1}+1}{3}.
\]
Set \(P_0=001011\).  Given \(P_r\), define
\begin{equation}
\label{eq:law1}
 N_r=\Interleave(P_r[2:|P_r|),P_r[0:|P_r|-1),1)
\end{equation}
and
\begin{equation}
\label{eq:law2}
 P_{r+1}=\Interleave(00N_r1,0N_r,0).
\end{equation}

\Needspace{5\baselineskip}
The lengths of the two families follow at once from the two laws.
\begin{lemma}

For every \(r\geq0\),
\[
 |P_r|=2a_r,
 \qquad |N_r|=4a_r-3.
\]
\end{lemma}

\begin{proof}
The base has length \(6=2a_0\).  If \(|P_r|=2a_r\), the two tapes in
\eqref{eq:law1} have total length \(4a_r-3\).  The two tapes in
\eqref{eq:law2} then have total length \(8a_r-2=2a_{r+1}\).
\end{proof}

The positions and word lengths fit without gaps:
\[
 u_0=4,
 \qquad
 v_r-u_r=2a_r=|P_r|,
 \qquad
 u_{r+1}-v_r=4a_r-3=|N_r|.
\]
Consequently, the half-open intervals \([u_r,v_r)\) and
\([v_r,u_{r+1})\), for \(r\geq0\), partition all increment indices
\(m\geq4\).

Let \(\mathrm{AP}(r)\) denote the following assertions.
\[
\begin{gathered}
 |P_r|=2a_r,
 \qquad Z(P_r)=O(P_r)=a_r,
 \qquad P_r[0:2)=00,\\
 h_{P_r}(0)=h_{P_r}(2a_r)=0,
 \qquad h_{P_r}(t)\geq1\quad(0<t<2a_r).
\end{gathered}
\tag{AP}
\]

\Needspace{5\baselineskip}
Under (AP), the letter counts and the terminal letters of both words are forced.
\begin{lemma}
\label{lem:terminal-counts}
If \(\mathrm{AP}(r)\) holds, then \(P_r\) ends with \(11\).  The word
\(N_r\) has \(2a_r-2\) zeros and \(2a_r-1\) ones, has final height \(-1\),
and ends with a one.
\end{lemma}

\begin{proof}
Interior positivity and the terminal height zero force the final step of
\(P_r\) to descend from height one.  The preceding letter cannot be zero,
since that would place the previous cut at height zero.  Thus the last two
letters are \(11\).

The first Law~1 tape loses the initial \(00\), while the second loses the
terminal one.  Their combined zero and one counts are therefore
\(2a_r-2\) and \(2a_r-1\).  Both tapes end in one.  Whichever tape is
exhausted last supplies a terminal one.
\end{proof}

\Needspace{5\baselineskip}
Neither law ever uses the fallback transition.
\begin{lemma}

If \(\mathrm{AP}(r)\) holds, neither interleaving
\eqref{eq:law1} nor \eqref{eq:law2} uses fallback.
\end{lemma}

\begin{proof}
Assume a first fallback.  All earlier reads obey the head formulas of
Lemma~\ref{lem:head-counts}.  The output that creates the current state has
not yet been counted by the head of the requested tape.  For a state-one
request before the last output, that state-creating one and the terminal one
are distinct uncounted letters.

For Law~1, the tape lengths are \(2a_r-2\) and \(2a_r-1\), equal to the
total zero and one counts of \(N_r\).  At a state-zero request, the first head
is at most \(Z(N_r)-1\).  At a state-one request, the second head is at most
\(1+O(N_r)-2\).  Both bounds are strictly below the corresponding tape
length.

For Law~2, the tape lengths are \(4a_r\) and \(4a_r-2\).  Their combined
letters give \(4a_r-1\) zeros and the same number of ones in \(P_{r+1}\).
Both tapes end with one, so the final output is one without using any
positivity assertion about \(P_{r+1}\).  A state-zero request has first head
at most \(1+(4a_r-2)=4a_r-1\).  A state-one request before the final output
has second head at most \(4a_r-3\).  These are the last valid indices of the
two tapes.  A first fallback is impossible.
\end{proof}

\Needspace{5\baselineskip}
Exact fit turns Law 1 into an exact height identity.
\begin{lemma}
\label{lem:law1-height}
Assume \(\mathrm{AP}(r)\).  If \(\xi_t,\upsilon_t\) are the two Law~1 head positions
after \(t\) outputs, then
\[
 h_{N_r}(t)=h_{P_r}(\xi_t+2)+h_{P_r}(\upsilon_t)-2.
\]
Moreover,
\[
 h_{N_r}(t)\geq0\quad(0\leq t<|N_r|),
 \qquad h_{N_r}(|N_r|)=-1,
\]
and
\[
 \max_t h_{N_r}(t)\leq2V_r-2.
\]
\end{lemma}

\begin{proof}
Height addition and exact fit give the displayed identity.  At \(t=0\), the
heads are \((\xi_0,\upsilon_0)=(0,0)\), and the formula reads
\(0=h_{P_r}(2)+h_{P_r}(0)-2\).  For \(t\geq1\), the initial state one and
Lemma~\ref{lem:head-counts} give
\[
 \xi_t=Z_{N_r}(t-1),
 \qquad
 \upsilon_t=1+O_{N_r}(t-1).
\]

Fix \(0<t<|N_r|\).  For an interior pair of heads, both \(P_r\) heights are
at least one.  If \(\xi_t+2=2a_r\), the first head formula shows that all
zeros of \(N_r\) have occurred, and
\[
 h_{N_r}(t)=4a_r-4-t\geq0.
\]
If \(\upsilon_t=2a_r-1\), then \(h_{P_r}(\upsilon_t)=1\) because \(P_r\)
ends with \(11\).  Exact fit prevents both tapes from being exhausted before
the final cut, so the first contribution is at least one.
The terminal value follows from Lemma~\ref{lem:terminal-counts}.  Replacing
both \(P_r\) heights by \(V_r\) gives the final bound.
\end{proof}

\Needspace{5\baselineskip}
The assertions (AP) propagate through all levels.
\begin{lemma}

The assertions \(\mathrm{AP}(r)\) hold for every \(r\geq0\).  In addition,
every \(N_r\) ends with \(11\).
\end{lemma}

\begin{proof}
The base word \(001011\) has interior heights \(1,2,1,2,1\).  Assume
\(\mathrm{AP}(r)\).  Lemmas~\ref{lem:terminal-counts} through
\ref{lem:law1-height}
apply.  Put \(S_0=00N_r1\) and \(S_1=0N_r\).  Their exact height profiles
satisfy
\[
 h_{S_i}(0)=h_{S_i}(|S_i|)=0,
 \qquad
 h_{S_i}(q)>0\quad(0<q<|S_i|),
 \qquad i\in\{0,1\}.
\]
For an interior cut of the output, the head pair is neither
\((0,0)\) nor \((|S_0|,|S_1|)\).  The pair \((0,|S_1|)\) is impossible
because the first read comes from \(S_0\).  The pair \((|S_0|,0)\) is also
impossible.  Before its last letter, \(S_0\) contains a one, which prescribes
a subsequent read from \(S_1\).  At least one head is therefore interior,
and the other contributes a nonnegative height.  Height addition gives a
strictly positive height for \(P_{r+1}\) at each interior cut.
Exact fit gives balance, and the first tape begins with \(00\).  This proves
\(\mathrm{AP}(r+1)\).

Both Law~1 tapes end with one.  If its final output came from the first tape,
the second tape would have been exhausted after emitting its own last one.
The next state would still request the second tape, contrary to exact fit.
The final output therefore comes from the second tape in state one, so the
last two outputs are \(11\).
\end{proof}

\section{Well-definedness}

\subsection{The slow sequences}

For any sequence \(X\), write
\[
 \Delta X(m)=X(m+1)-X(m).
\]
Set
\[
 \mathfrak a(1),\mathfrak a(2),\mathfrak a(3),\mathfrak a(4)=1,1,2,3,
 \qquad
 \mathfrak b(1),\mathfrak b(2),\mathfrak b(3),\mathfrak b(4)=1,2,2,3.
\]
For \(u_r\leq m<v_r\), define
\[
 \Delta \mathfrak a(m)=P_r[m-u_r],
 \qquad \Delta \mathfrak b(m)=1-P_r[m-u_r].
\]
For \(x=v_r+s\), where \(0\leq s<|N_r|\), define
\[
 \Delta \mathfrak b(v_r+s)=N_r[s],
 \qquad \Delta \mathfrak a(v_r)=1,
\]
and
\[
 \Delta \mathfrak a(v_r+s)=1-N_r[s-1]
 \qquad(1\leq s<|N_r|).
\]

\Needspace{5\baselineskip}
All increments are binary, and the two sequences agree at the arch boundaries.
\begin{lemma}
\label{lem:endpoints}
All increments of \(\mathfrak a\) and \(\mathfrak b\) lie in \(\{0,1\}\).  For every \(m\geq1\),
\[
 1\leq \mathfrak a(m)\leq m,
 \qquad 1\leq \mathfrak b(m)\leq m.
\]
For every \(r\geq0\),
\[
 \mathfrak a(u_r)=\mathfrak b(u_r)=a_r,
 \qquad \mathfrak a(v_r)=\mathfrak b(v_r)=2a_r,
\]
and
\[
 \mathfrak a(u_{r+1})=\mathfrak b(u_{r+1})=a_{r+1}.
\]
\end{lemma}

\begin{proof}
The three initial increments are binary.  Every later increment is a letter
of some \(P_r\) or \(N_r\), or the complement of such a letter.  The bounds follow by
induction from \(\mathfrak a(1)=\mathfrak b(1)=1\).

For the endpoints, argue by induction on \(r\).  The base is
\(u_0=4\) and \(\mathfrak a(4)=\mathfrak b(4)=3=a_0\).  Assume the equality
at \(u_r\).  The balanced word \(P_r\)
raises both sequences by \(a_r\) on the positive interval.  On the negative
interval, \(\mathfrak b\) increases by the number \(2a_r-1\) of ones in \(N_r\).  The
increase of \(\mathfrak a\) is
\[
 1+\sum_{s=1}^{|N_r|-1}(1-N_r[s-1])=2a_r-1,
\]
because the final one of \(N_r\) is excluded from the sum.  Both sequences
therefore reach \(4a_r-1=a_{r+1}\).
\end{proof}

Put
\[
 \delta(m)=\mathfrak b(m)-\mathfrak a(m).
\]
On the positive interval,
\begin{equation*}
 \delta(u_r+t)=h_{P_r}(t)
 \qquad(0\leq t\leq2a_r).
\end{equation*}
Thus \([u_r,v_r]\) is a positive arch and its first and last cuts are its
only zero-height cuts.

Figure~\ref{fig:positive-walk} illustrates this identity at level \(r=2\).

\begin{figure}[htbp]
\centering
\includegraphics[width=\textwidth]{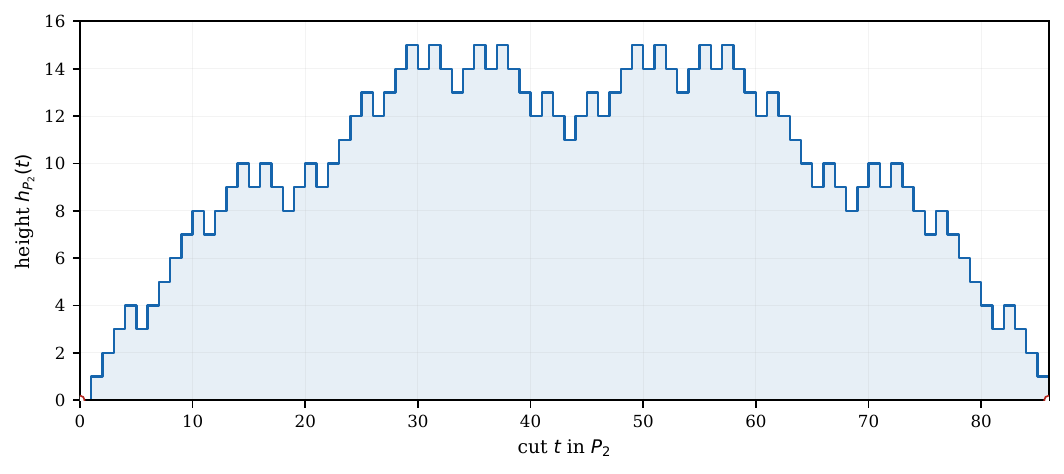}
\caption{The positive-arch word \(P_2\) and its height path.}
\label{fig:positive-walk}
\end{figure}

On the negative interval, direct summation gives
\begin{equation}
\label{eq:negative-dictionary}
 -\delta(v_r+t)=h_{N_r}(t)+N_r[t-1]
 =\max\{h_{N_r}(t-1),h_{N_r}(t)\}
\end{equation}
for \(1\leq t\leq|N_r|\).  The formula never reads \(N_r[|N_r|]\).
Since \(N_r[0]=0\), \(N_r\) ends with \(11\), and its final two heights are
zero and \(-1\), the two boundary depths are one and zero.

Figure~\ref{fig:delta-arches} shows the beginning of the complete signed walk.

\begin{figure}[htbp]
\centering
\includegraphics[width=\textwidth]{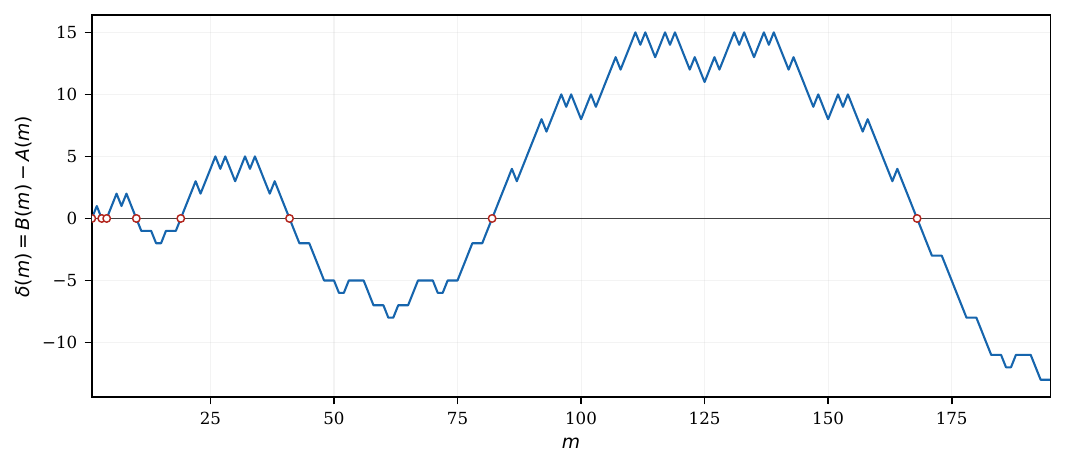}
\caption{The walk \(\delta(m)=\mathfrak b(m)-\mathfrak a(m)\) for
\(1\leq m\leq195\). Marked cuts separate successive arches.}
\label{fig:delta-arches}
\end{figure}

\Needspace{5\baselineskip}
Four linear identities tie $\mathfrak a$ and $\mathfrak b$ along the arches.
\begin{lemma}

For every \(r\geq0\),
\begin{equation}
\label{eq:ZP}
 \mathfrak a(m)+\mathfrak b(m)=m+r+2
 \qquad(u_r\leq m\leq v_r),
\end{equation}
and
\begin{equation}
\label{eq:ZN}
 \mathfrak b(x)+\mathfrak a(x+1)=x+r+3
 \qquad(v_r\leq x<u_{r+1}).
\end{equation}
For \(r\geq1\),
\begin{equation}
\label{eq:bridge-left}
 \mathfrak a(z+1)+\mathfrak b(z)=z+r+2
 \qquad(v_{r-1}-2\leq z\leq u_r),
\end{equation}
and for \(r\geq0\),
\begin{equation}
\label{eq:bridge-right}
 \mathfrak a(z+1)+\mathfrak b(z)=z+r+3
 \qquad(v_r-1\leq z<u_{r+1}).
\end{equation}
\end{lemma}

\begin{proof}
On a positive arch, \(\Delta\mathfrak a+\Delta\mathfrak b=1\).  The first endpoint identity
in Lemma~\ref{lem:endpoints} gives \eqref{eq:ZP}.  On a negative arch,
\(\Delta \mathfrak b(x)+\Delta \mathfrak a(x+1)=1\), including its last step because \(N_r\)
ends in one and \(P_{r+1}\) begins in zero.  The value at \(v_r\) gives
\eqref{eq:ZN}.

Equation \eqref{eq:bridge-right} adds the cut \(v_r-1\) to \eqref{eq:ZN}.
The last letter of \(P_r\) is one, so \(\Delta \mathfrak b(v_r-1)=0\).
Together with the endpoint values, this gives
\[
 \mathfrak a(v_r)+\mathfrak b(v_r-1)=4a_r=v_r+r+2,
\]
which is \eqref{eq:bridge-right} at that extra cut.

For \eqref{eq:bridge-left}, equation \eqref{eq:ZN} at level \(r-1\) covers
\(v_{r-1}\leq z<u_r\).  Put \(a=a_{r-1}\) and \(v=v_{r-1}\).  The terminal
\(11\) of \(P_{r-1}\) gives
\[
 \Delta\mathfrak a(v-2)=\Delta\mathfrak a(v-1)=1,
 \qquad
 \Delta\mathfrak b(v-2)=\Delta\mathfrak b(v-1)=0.
\]
Since \(\mathfrak a(v)=\mathfrak b(v)=2a\), the two additional left values
are
\[
 \mathfrak a(v)+\mathfrak b(v-1)=4a=(v-1)+r+2,
\]
and
\[
 \mathfrak a(v-1)+\mathfrak b(v-2)=4a-1=(v-2)+r+2.
\]
Finally, \(P_r[0]=0\) gives
\(\mathfrak a(u_r+1)=\mathfrak a(u_r)=a_r\), and hence
\[
 \mathfrak a(u_r+1)+\mathfrak b(u_r)=2a_r=u_r+r+2.
\]
These are exactly the remaining three boundary values.
\end{proof}

\subsection{The coupled recurrences}

\Needspace{5\baselineskip}
A head abandoned by the machine is parked immediately after a one and stays there during the ensuing reads.
\begin{lemma}
\label{lem:parked-head-invariant}
In an exact-fit interleaving, immediately after a noninitial transition into
state one, the head of the state-zero tape is immediately to the right of a
one on that tape.  It remains there throughout the ensuing consecutive
state-one reads.
\end{lemma}

\begin{proof}
The output that causes such a transition is a one read from the state-zero
tape, whose head has just advanced past it.  Subsequent state-one reads move
only the other head.
\end{proof}

\Needspace{5\baselineskip}
The head positions translate into coordinates of the two slow sequences.
\begin{lemma}

On the negative arch \(v_r\leq x<u_{r+1}\), put
\[
 k=x+1-\mathfrak a(x),
 \qquad j=x+1-\mathfrak b(x-1),
 \qquad d=\Delta \mathfrak b(x-1).
\]
If \(x=v_r+s\) with \(s\geq1\), then, for either value of \(d\),
\begin{equation}
\label{eq:negative-head-coordinates}
 j=u_r+2+Z_{N_r}(s-1),
 \qquad
 k=u_r+1+O_{N_r}(s-1).
\end{equation}
Then
\begin{align}
\label{eq:L1B}
 \Delta \mathfrak b(x)&=(1-d)\Delta \mathfrak a(j)+d\Delta \mathfrak a(k),\\
\label{eq:L1A}
 \Delta \mathfrak a(x)&=1-d\Delta \mathfrak a(j-1).
\end{align}
On the positive arch \(u_{r+1}\leq m<v_{r+1}\), put
\[
 \ell=m-\mathfrak a(m-1),
 \qquad j=m-\mathfrak b(m-1),
 \qquad c=\Delta \mathfrak a(m-1).
\]
Then
\begin{align}
\label{eq:L2A}
 \Delta \mathfrak a(m)&=(1-c)\Delta \mathfrak b(\ell)+c\Delta \mathfrak b(j),\\
\label{eq:L2B}
 \Delta \mathfrak b(m)&=1-\Delta \mathfrak a(m).
\end{align}
If \(c=1\), which occurs only after the first output, the parked Law~2 head satisfies
\begin{equation}
\label{eq:parked-positive}
 \Delta \mathfrak b(m-\mathfrak a(m-1)-1)=1.
\end{equation}
If \(d=1\), the parked Law~1 head satisfies
\begin{equation}
\label{eq:parked-negative}
 \Delta \mathfrak a(j-1)=1.
\end{equation}
\end{lemma}

\begin{proof}
Exact fit makes each head position equal to the number of outputs previously
charged to its tape.

For Law~1, the tapes are \(X=P_r[2:|P_r|)\) and
\(Y=P_r[0:|P_r|-1)\), with
\[
 X[p]=\Delta\mathfrak a(u_r+2+p),
 \qquad
 Y[q]=\Delta\mathfrak a(u_r+q).
\]
Write \(x=v_r+s\).  For \(s\geq1\), direct summation of both negative
increment sequences gives the two unconditional identities
\eqref{eq:negative-head-coordinates}.  The state before the output is
\(d=N_r[s-1]=\Delta\mathfrak b(x-1)\).  When \(d=0\), the \(X\)-head is
\(Z_{N_r}(s-1)\), and the symbol read is \(\Delta\mathfrak a(j)\).  When
\(d=1\), the \(Y\)-head is \(1+O_{N_r}(s-1)\), and the symbol read is
\(\Delta\mathfrak a(k)\).  At \(s=0\), one has
\(d=0\) and \(j=u_r+1\).  The initial output is \(P_r[0]=0\), while the
formula reads \(\Delta\mathfrak a(u_r+1)=P_r[1]=0\).  This proves
\eqref{eq:L1B} at every negative increment.

If \(d=0\), the definition gives \(\Delta\mathfrak a(x)=1\).  If \(d=1\),
Lemma~\ref{lem:parked-head-invariant} says that the parked \(X\)-head is just
past a one, which in the displayed coordinates is
\(\Delta\mathfrak a(j-1)=1\).  The negative definition then gives
\(\Delta\mathfrak a(x)=0\).  These two cases prove \eqref{eq:L1A} and
\eqref{eq:parked-negative}.

For Law~2, put \(S_0=00N_r1\) and \(S_1=0N_r\).  The exact coordinate
identities are
\[
 S_0[p]=\Delta\mathfrak b(v_r-2+p),
 \qquad
 S_1[q]=\Delta\mathfrak b(v_r-1+q).
\]
Write \(m=u_{r+1}+s\).  At \(s=0\), the first output is zero, so
\(\Delta\mathfrak a(u_{r+1})=0\).  The terminal \(11\) of \(N_r\) gives
\(c=\Delta\mathfrak a(u_{r+1}-1)=0\), and the endpoint identities give
\[
 \ell=u_{r+1}-\mathfrak a(u_{r+1}-1)=v_r-2.
\]
The terminal \(11\) of \(P_r\) gives
\(\Delta\mathfrak b(v_r-2)=0\), so \eqref{eq:L2A} holds at this boundary.

For \(s\geq1\), the state is
\(c=P_{r+1}[s-1]=\Delta\mathfrak a(m-1)\).  If \(c=0\), the \(S_0\)-head
is \(1+Z_{P_{r+1}}(s-1)\), and
\[
 \ell=m-\mathfrak a(m-1)
 =v_r-2+1+Z_{P_{r+1}}(s-1).
\]
If \(c=1\), the \(S_1\)-head is \(O_{P_{r+1}}(s-1)\), and
\[
 j=m-\mathfrak b(m-1)=v_r-1+O_{P_{r+1}}(s-1).
\]
The emitted symbol is \(\Delta\mathfrak a(m)\), which proves
\eqref{eq:L2A}.  The positive-arch definition proves \eqref{eq:L2B}.
When \(c=1\), Lemma~\ref{lem:parked-head-invariant} places the \(S_0\)-head
immediately after a one.  Its coordinate in the full input word is \(\ell\), so
\(\Delta\mathfrak b(\ell-1)=1\), which is
\eqref{eq:parked-positive}.
\end{proof}

\Needspace{5\baselineskip}
The construction meets the two recurrences of the parity split.
\begin{theorem}
\label{thm:coupled-recurrences}
The constructed sequences \(\mathfrak a,\mathfrak b\) satisfy \eqref{eq:RA} and \eqref{eq:RB}
for every \(m\geq2\).
\end{theorem}

\begin{proof}
The values through \(v_0=10\) are
\[
 \mathfrak a=(1,1,2,3,3,3,4,4,5,6),
 \qquad
 \mathfrak b=(1,2,2,3,4,5,5,6,6,6).
\]
Direct substitution verifies both recurrences on this finite initial segment.
More explicitly, their right-hand sides for \(2\leq m\leq10\) are
\[
\begin{array}{c|ccccccccc}
 m&2&3&4&5&6&7&8&9&10\\ \hline
 \mathrm{RA}&1&2&3&3&3&4&4&5&6\\
 \mathrm{RB}&2&2&3&4&5&5&6&6&6
\end{array}
\]
and equal the displayed values of \(\mathfrak a(m)\) and \(\mathfrak b(m)\).
The induction runs along the alternating arches.  At each level, RA at the left
endpoint propagates RA across the positive arch, the bridge identities give
RB there, Law~1 propagates RB across the negative arch, and the bridges give
RA at the next left endpoint.

Fix a positive arch of level \(r\).  RA is known at \(u_r\), either from the
initial table or from the preceding negative arch.  Define
\[
 \operatorname{rhs}_{\mathfrak a}(m)
 =\mathfrak b(m-\mathfrak b(m-1))+\mathfrak b(m-\mathfrak a(m-1))-1.
\]
For \(u_r\leq m<v_r\), put \(c=\Delta \mathfrak a(m-1)\).  One has
\[
 \Delta\mathfrak b(m-1)=1-c.
\]
For \(m>u_r\), this is the positive-arch definition.  At \(m=u_r\), with
\(r\geq1\), the terminal \(11\) of \(N_{r-1}\) gives
\[
 \Delta\mathfrak b(u_r-1)=1,
 \qquad
 \Delta\mathfrak a(u_r-1)=0.
\]
The case \(r=0\) is in the initial table.  Comparison of two consecutive values gives
\[
 \operatorname{rhs}_{\mathfrak a}(m+1)-\operatorname{rhs}_{\mathfrak a}(m)
 =(1-c)\Delta \mathfrak b(m-\mathfrak a(m-1))
 +c\Delta \mathfrak b(m-\mathfrak b(m-1)).
\]
This is \(\Delta \mathfrak a(m)\) by \eqref{eq:L2A}.  Thus \eqref{eq:RA}
propagates across the positive arch.

For \(r\geq1\), the value of \eqref{eq:RB} at \(m=u_r\) has already been
obtained at the end of the preceding negative arch.  For
\(u_r<m\leq v_r\), put
\[
 x_1=m-\mathfrak b(m-1),
 \qquad x_2=m-\mathfrak a(m-1),
 \qquad \varepsilon=\Delta \mathfrak a(m-1).
\]
Equation \eqref{eq:ZP} at \(m-1\) gives
\[
 x_1=\mathfrak a(m-1)-r-1,
 \qquad
 x_2-\varepsilon=m-\mathfrak a(m).
\]
Using the endpoint values and
\(a_r-r-1=v_{r-1}-1\), \(2a_r-r-2=u_r\), one obtains
\[
 v_{r-1}-1\leq x_1\leq u_r,
 \qquad
 v_{r-1}-1\leq x_2-\varepsilon\leq u_r.
\]
Thus the bridge identity \eqref{eq:bridge-left} gives
\[
 \mathfrak a(x_1+1)=x_1+r+2-\mathfrak b(x_1)
\]
and
\[
 \mathfrak a(x_2+1-\varepsilon)=x_2+r+2-\mathfrak b(x_2).
\]
For the second identity, apply the bridge at \(x_2\) when \(\varepsilon=0\).
When \(\varepsilon=1\) and \(m<v_r\), equation
\eqref{eq:parked-positive} gives
\(\mathfrak b(x_2)=\mathfrak b(x_2-1)+1\), and the bridge at \(x_2-1\)
gives the same formula.  At the remaining cut \(m=v_r\), the terminal one of
\(P_r\) gives
\[
 x_2=v_r-\mathfrak a(v_r-1)=u_r+1,
 \qquad
 \Delta\mathfrak b(x_2-1)=\Delta\mathfrak b(u_r)=1-P_r[0]=1,
\]
so the bridge at \(u_r=x_2-1\) gives the formula there as well.  Adding these
equations, using \eqref{eq:RA}, and using
\(x_1+x_2=m-r-1\), obtained from \eqref{eq:ZP} at \(m-1\), gives
\[
 \mathfrak a(m+1-\mathfrak a(m))+\mathfrak a(m+1-\mathfrak b(m-1))=m+r+2-\mathfrak a(m)=\mathfrak b(m).
\]

On the following negative arch, define
\[
 \operatorname{rhs}_{\mathfrak b}(x)
 =\mathfrak a(x+1-\mathfrak a(x))+\mathfrak a(x+1-\mathfrak b(x-1)).
\]
Put
\[
 k=x+1-\mathfrak a(x),
 \qquad
 j=x+1-\mathfrak b(x-1),
 \qquad
 d=\Delta\mathfrak b(x-1).
\]
When \(x\) increases by one, the two arguments increase by
\(1-\Delta\mathfrak a(x)\) and \(1-d\), respectively.  If \(d=0\),
equation \eqref{eq:L1A} gives \(\Delta\mathfrak a(x)=1\), so only the second
argument advances, and
\(\operatorname{rhs}_{\mathfrak b}(x+1)-\operatorname{rhs}_{\mathfrak b}(x)
=\Delta \mathfrak a(j)=\Delta \mathfrak b(x)\).  If \(d=1\),
\eqref{eq:parked-negative} and \eqref{eq:L1A} show that its first argument
advances, and
\(\operatorname{rhs}_{\mathfrak b}(x+1)-\operatorname{rhs}_{\mathfrak b}(x)
=\Delta \mathfrak a(k)=\Delta \mathfrak b(x)\).  Hence \eqref{eq:RB}
propagates.

To obtain \eqref{eq:RA} at \(m=x+1\), equations
\eqref{eq:ZN} and \eqref{eq:RB} give
\[
 \mathfrak a(x+1)=x+r+3-\mathfrak a(k)-\mathfrak a(j).
\]
The explicit head coordinates, together with
\(k=j=u_r+1\) and \(d=0\) at the boundary \(x=v_r\), give
\[
 k,\ j-d\in[u_r,v_r],
\]
and, when \(d=1\), also \(j-1\in[u_r,v_r]\).  Thus equation \eqref{eq:ZP},
together with \eqref{eq:parked-negative} when \(d=1\), gives
\[
 \mathfrak b(k)=k+r+2-\mathfrak a(k),
 \qquad \mathfrak b(j-d)=j+r+2-\mathfrak a(j).
\]
Equation \eqref{eq:bridge-right} at \(x-1\) yields \(j+k=x-r\).  Therefore
\[
 \mathfrak b(j-d)+\mathfrak b(k)-1=\mathfrak a(x+1),
\]
which is \eqref{eq:RA}.  The alternating arches exhaust all indices.
\end{proof}

\Needspace{5\baselineskip}
\begin{theorem}
\label{thm:well-definedness}

The recurrence \eqref{eq:Qt} is well-defined for every positive integer.
Moreover,
\[
 1\leq\Qt(n)\leq n\qquad(n\geq1).
\]
\end{theorem}

\begin{proof}
Define
\[
 \Qt(2m-1)=2\mathfrak a(m)-1,
 \qquad \Qt(2m)=2\mathfrak b(m)-1.
\]
The bounds in Lemma~\ref{lem:endpoints} give the four exact inequalities
\[
 1\leq m-\mathfrak a(m-1),\ m-\mathfrak b(m-1)\leq m-1,
\]
\[
 1\leq m+1-\mathfrak a(m)\leq m,
 \qquad
 2\leq m+1-\mathfrak b(m-1)\leq m.
\]
At the odd index \(2m-1\), the recurrence arguments are twice the first two
quantities and hence lie in \([2,2m-2]\).  At the even index \(2m\), they
are twice the last two quantities minus one and hence lie in
\([1,2m-1]\).  Thus every evaluation of \eqref{eq:Qt} uses positive earlier
arguments.  Theorem
\ref{thm:coupled-recurrences} and Lemma~\ref{lem:parity-split} prove the
recurrence at every odd and even index.  Determinism gives uniqueness.  The
bound on \(\Qt\) follows from \(1\leq \mathfrak a(m),\mathfrak b(m)\leq m\).
\end{proof}

\section{Words and forests}

\subsection{The word transform \texorpdfstring{\(T\)}{T}}
\label{sec:transform-T}

For a Dyck word \(W\), define
\[
 T(W)=\Interleave(0W1,W,0),
 \qquad
 \mathcal C(W)=T(W)[1:|T(W)|-1),
\]
and
\[
 \Lam(W)=T^2(W)[1:|T^2(W)|-1).
\]
\Needspace{5\baselineskip}
The zero data of an image under $T$ admit a closed form.
\begin{lemma}
\label{lem:zero-data}
Let \(W\) be a Dyck word of length \(2n\), and let \(p'_k,q'_k,\omega'_k\)
denote the zero data of \(T(W)\).  The interleaving defining \(T(W)\) has
exact fit.  For \(1\leq k\leq2n+1\),
\begin{equation}
\label{eq:T-zero-data}
 q'_k=\theta_W(k-1),
 \qquad p'_k=k+\theta_W(k-1).
\end{equation}
For every \(0\leq t<2n\),
\begin{equation}
\label{eq:T-gap-data}
 \omega'_1=1,
 \qquad
 \omega'_{t+2}=
 \begin{cases}
 1,&W_t=0,\\
 1+\omega_{O_W(t+1)},&W_t=1,
 \end{cases}
\end{equation}
The gaps of \(\mathcal C(W)\) are obtained by removing
the first term of \eqref{eq:T-gap-data}.
\end{lemma}

\begin{proof}
At a cut immediately following a zero of \(T(W)\), the state is zero.  The
numbers of letters consumed from \(0W1\) and \(W\) equal the zero and one
counts of the emitted prefix.  This follows by induction from one such cut to
the next.  A factor \(0\) advances the first tape once.  A factor \(1^a0\)
advances the first tape once and the second tape \(a\) times.

After the \(k\)-th zero, the first head has consumed \(k\) letters.  It has
read the initial zero of \(0W1\) and the prefix \(W[0:k-1)\).  Every one in
that prefix makes the second tape advance through the following zero.  The
second head is therefore at
\(p_{O_W(k-1)}=\theta_W(k-1)\).  This is the first formula in
\eqref{eq:T-zero-data}.  The second follows from \(p'_k=k+q'_k\).

At the last zero, the first tape has only its terminal one left.  Emitting it
puts the transducer in state one, after which the terminal ones of the second
tape are consumed.  Both tapes are exhausted and no fallback occurs.

Subtracting consecutive cuts in \eqref{eq:T-zero-data} gives
\eqref{eq:T-gap-data}.  A zero leaves \(O_W\) unchanged.  At the \(s\)-th one,
\(\theta_W\) increases by \(p_s-p_{s-1}=\omega_s\).  Removing the first zero of
\(T(W)\) removes exactly its first gap.
\end{proof}

\Needspace{5\baselineskip}
The transform preserves Dyck words.
\begin{lemma}
\label{lem:T-dyck}
If \(W\) is a Dyck word, then \(T(W)\) is a primitive Dyck word and
\(\mathcal C(W)\) is a Dyck word.
\end{lemma}

\begin{proof}
The two tapes contain the same total number of zeros and ones.  Lemma
\ref{lem:zero-data} gives exact consumption.  Since a Dyck prefix has at
least as many zeros as ones,
\[
 \theta_W(t)\leq t.
\]
Thus \(q'_k\leq k-1\), which is the zero-cut form of the Dyck condition.  If
\(t\geq1\), equality \(\theta_W(t)=t\) would force the last letter of the
prefix to be zero and the preceding height to be \(-1\).  Hence
\(\theta_W(t)<t\).  The word \(T(W)\) is strictly positive between its
endpoints.  Removing its first zero and its last one leaves a Dyck word.
\end{proof}

Define the orbital cores
\begin{equation*}
 A_r=P_r[1:|P_r|-1).
\end{equation*}

\Needspace{5\baselineskip}
Two applications of $T$ reproduce the two laws.
\begin{lemma}
\label{lem:normal-form}
For every \(r\geq0\),
\[
 T(A_r)=0N_r,
 \qquad T^2(A_r)=P_{r+1},
 \qquad A_{r+1}=\Lam(A_r).
\]
\end{lemma}

\begin{proof}
Consider the word \(0N_r\) with initial state zero.  Its first two letters
come from the first Law~1 input and are \(00=P_r[0:2)\).  After the initial
external zero, the remaining letters assigned to state zero are exactly
\(P_r[2:|P_r|)\).  Together they form \(P_r=0A_r1\).  The initial Law~1
letter assigned to state one has been reassigned to state zero by the
external initial zero.  The remaining state-one letters are therefore
\(P_r[1:|P_r|-1)=A_r\).  The inverse description of interleaving gives
\[
 0N_r=\Interleave(P_r,A_r,0)=T(A_r).
\]
Law~2 is
\[
 P_{r+1}=\Interleave(0(0N_r)1,0N_r,0)=T(0N_r)=T^2(A_r).
\]
Removing the envelope proves the last identity.
\end{proof}

\subsection{Gaps and symmetry}
\label{sec:gaps-symmetry}

For \(L\geq1\) and \(1\leq i\leq L\), define
\begin{equation}
\label{eq:sigma-gap}
 \sigma_L(i)=
 \begin{cases}
 2i,&2i\leq L,\\
 2(L-i)+1,&2i>L,
 \end{cases}
\end{equation}
and
\begin{equation*}
 \eta_L(i)=
 \begin{cases}
 2i-1,&2i\leq L+1,\\
 2(L-i+1),&2i>L+1.
 \end{cases}
\end{equation*}
Set \(\sigma_0=\eta_0=()\).  Both words are permutations of
\(1,\ldots,L\), and
\begin{equation}
\label{eq:gap-toggle-identities}
 \eta_L=(1,1+\sigma_{L-1}),
 \qquad \sigma_L=(1+\eta_{L-1},1).
\end{equation}
Put
\[
 \varphi(L)=\sigma_L-1,
 \qquad \varphi(0)=(),
\]
and extend \(\varphi\) to words by concatenation.

For a Dyck word with degree word \(c=(c_1,\ldots,c_n)\), define the source
factor of the \(m\)-th zero by
\[
 \mathcal B_m=(c_{q_{m-1}+1},\ldots,c_{q_m}).
\]
It is empty when \(c_m=0\).
Its gap version is
\[
 \mathcal G_m=(\omega_{q_{m-1}+1},\ldots,\omega_{q_m})
 =1+\mathcal B_m.
\]

\begin{definition}
The word \(W\) has property \(\mathsf S\) if
\begin{equation}
\label{eq:property-S}
 \mathcal B_m=\varphi(c_m)
 \qquad(1\leq m\leq n).
\end{equation}
Equivalently, \(\mathcal G_m=\sigma_{c_m}\).  It has the companion property
\(\mathsf R\) if \(\mathcal G_m=\eta_{c_m}\) at every \(m\).
\end{definition}

\Needspace{5\baselineskip}
The properties $\mathsf S$ and $\mathsf R$ toggle under the transform.
\begin{lemma}
\label{lem:profile-toggle}
For every Dyck word \(W\),
\[
 \mathsf S(W)\Longrightarrow\mathsf R(T(W)),
 \qquad
 \mathsf R(W)\Longrightarrow\mathsf S(\mathcal C(W)).
\]
\end{lemma}

\begin{proof}
A gap of \(T(W)\) produced by a zero of \(W\) equals one and has an empty
source factor.  A gap produced by the \(s\)-th one has value \(1+\omega_s(W)\).
Lemma~\ref{lem:zero-data} shows that its source factor is
\[
 (1,1+\mathcal G_s(W)).
\]
Under \(\mathsf S(W)\), this is \(\eta_{\omega_s(W)}\) by
\eqref{eq:gap-toggle-identities}.  After removal of the enveloping gap, the
factor becomes
\[
 (1+\mathcal G_s(W),1).
\]
Under \(\mathsf R(W)\), it is \(\sigma_{\omega_s(W)}\).  Empty factors and the
first and last gaps are included in this calculation.
\end{proof}

\Needspace{5\baselineskip}
The orbit inherits $\mathsf S$ from its base.
\begin{corollary}
\label{cor:orbital-S}
Every core \(A_r\) is a Dyck word satisfying \(\mathsf S\).
\end{corollary}

\begin{proof}
The degree word of \(A_0=0101\) is \((0,1)\), and its two source factors are
\(\varphi(0)\) and \(\varphi(1)=(0)\).  Lemma
\ref{lem:profile-toggle}, Lemma~\ref{lem:normal-form}, and Lemma
\ref{lem:T-dyck} give the induction.
\end{proof}

For a binary word \(W\), write
\[
 W^{\Star}=\overline{\rev(W)}.
\]

\Needspace{5\baselineskip}
The transform commutes with reversal-complementation.
\begin{lemma}
\label{lem:T-star}
For every Dyck word \(W\),
\[
 T(W^{\Star})=T(W)^{\Star},
 \qquad
 \mathcal C(W^{\Star})=\mathcal C(W)^{\Star}.
\]
\end{lemma}

\begin{proof}
Let \(W\) have length \(2n\), let \(H=2n+1\), and put \(Z=T(W)\).  Denote by
\(s_r^Z\) the cut after the \(r\)-th one of \(Z\).  Reading the two heads
backward from the terminal cut gives
\begin{equation}
\label{eq:dual-one-cuts}
 s_r^Z=H+r-\theta_{W^{\Star}}(H-r)
 \qquad(1\leq r\leq H).
\end{equation}
Indeed, both sides agree at \(r=H\).  Their backward increments are zero
when the corresponding letter of \(W^{\Star}\) is zero.  When it is the
\(a\)-th one, both increments equal the gap \(p_a-p_{a-1}\).  This proves
\eqref{eq:dual-one-cuts} by backward induction.

The \(k\)-th zero of \(Z^{\Star}\) is the reflection of the
\((H+1-k)\)-th one of \(Z\).  With the terminal cut \(2H+1\),
\[
 p_k^{Z^{\Star}}
 =2H+1-s_{H+1-k}^Z
 =k+\theta_{W^{\Star}}(k-1).
\]
Lemma~\ref{lem:zero-data} gives the same zero cuts for \(T(W^{\Star})\).
Balanced words are determined by their zero cuts, so the first identity
holds.  Removing the first and last letters gives the second.
\end{proof}

\Needspace{5\baselineskip}
Anti-palindromicity propagates along the orbit.
\begin{corollary}
\label{cor:orbital-anti}
For every \(r\geq0\),
\[
 A_r=A_r^{\Star}.
\]
\end{corollary}

\begin{proof}
The base \(0101\) is fixed by the involution.  Lemma~\ref{lem:T-star} and
Lemma~\ref{lem:normal-form} give the induction.
\end{proof}

\subsection{Plane forests}

The degree coordinates satisfy
\begin{equation}
\label{eq:degree-expansion}
 W=1^{c_1}0\,1^{c_2}0\cdots1^{c_n}0\,1^{d_n}.
\end{equation}
The Dyck condition is \(q_m\leq m-1\) for \(1\leq m\leq n\).

\Needspace{5\baselineskip}
A word satisfying $\mathsf S$ is generated by its root word through an explicit grammar.
\begin{lemma}
\label{lem:root-grammar}
Let \(W\) be a nonempty Dyck word of length \(2n\) satisfying \(\mathsf S\).
Give vertex \(m\) the children
\[
 (q_{m-1},q_m]=\{q_{m-1}+1,\ldots,q_m\}.
\]
The roots are \(q_n+1,\ldots,n\).  Put
\[
 \varrho=d_n=n-q_n,
 \qquad R=(c_{q_n+1},\ldots,c_n).
\]
Then
\begin{equation}
\label{eq:root-fixed-point}
 c=\varphi(c)R.
\end{equation}
If \(M=\max R\), then
\begin{equation}
\label{eq:root-grammar}
 c=\varphi^M(R)\varphi^{M-1}(R)\cdots\varphi(R)R.
\end{equation}
Conversely, every nonempty root word \(R\) defines through
\eqref{eq:root-grammar} a unique Dyck word satisfying \(\mathsf S\).  Its
half-length is
\begin{equation}
\label{eq:root-weight}
 n=\sum_{x\in R}2^x.
\end{equation}
\end{lemma}

\begin{proof}
Concatenating \eqref{eq:property-S} for \(m=1,\ldots,n\) covers the indices
from one through \(q_n\).  The remaining \(\varrho\) degrees form \(R\),
which proves \eqref{eq:root-fixed-point}.  If the maximal degree of \(c\) is
positive, the maximal degree of \(\varphi(c)\) is one smaller.  Hence the
maximum of \(c\) equals \(M\).  If this maximum is zero, then \(c=R\), the
root word is nonempty, and \(M=0\), so the asserted grammar also holds.
Otherwise, iterating \eqref{eq:root-fixed-point} until
\(\varphi^{M+1}(c)\) is empty gives \eqref{eq:root-grammar}.

Conversely, make a rooted tree of type \(L\) whose root has degree \(L\) and
whose ordered children have types \(\varphi(L)\).  Types decrease along every
edge.  In the order \eqref{eq:root-grammar}, each child precedes its parent.
The forest induced by its first \(m\) vertices has at least one component, so
it has at most \(m-1\) edges.  Hence \(q_m\leq m-1\), which proves the Dyck
condition.  The source factors have the required types by construction.

The forest has \(n\) vertices and \(|R|\) components.  It consequently has
\(n-|R|\) edges, so
\[
 \sum_{m=1}^n c_m=n-|R|,
 \qquad d_n=n-q_n=|R|.
\]
This verifies both the terminal suffix in \eqref{eq:degree-expansion} and the
balance of the reconstructed binary word.  If \(N_L\) is the size of a tree
of type \(L\), then
\[
 N_0=1,
 \qquad N_L=1+\sum_{j=0}^{L-1}N_j.
\]
Thus \(N_L=2^L\).  Summing over the roots gives
\eqref{eq:root-weight}.  The degree expansion \eqref{eq:degree-expansion}
then gives a unique binary word.
\end{proof}

\begin{definition}
The unique Dyck word satisfying \(\mathsf S\) obtained from a nonempty root
word \(R\) in Lemma~\ref{lem:root-grammar} is the \emph{canonical word of
\(R\)}.  Its plane forest is the \emph{canonical forest of \(R\)}.
\end{definition}

For a single letter \(x\geq0\), put
\(\ell_k(x)=|\varphi^k(x)|\).  Directly from the definition of \(\varphi\),
\[
 \ell_0(x)=1,
 \qquad
 \ell_{k+1}(x)=\sum_{y=0}^{x-1}\ell_k(y),
\]
where an empty sum is zero.  Induction and the hockey-stick identity give
\(\ell_k(x)=\binom{x}{k}\), with the convention that the binomial coefficient
vanishes for \(k>x\).  Thus
\begin{equation}
\label{eq:phi-binomial-length}
 |\varphi^k(L)|=\binom Lk
\end{equation}
for all \(L,k\geq0\).  This identity will be used both in the root
classification and in the Gaussian layer analysis.

\subsection{Root words}

If \(c=0^n\), the run vector of \eqref{eq:degree-expansion} is, by convention,
\((n,n)\).  Otherwise let \(j_1<\cdots<j_s\) be the indices at which
\(c_{j_i}>0\).  The run lengths in \eqref{eq:degree-expansion} are
\begin{equation}
\label{eq:run-vector}
 (j_1-1,c_{j_1},j_2-j_1,c_{j_2},\ldots,
 c_{j_s},n-j_s+1,\varrho).
\end{equation}
The word is anti-palindromic exactly when this vector is a palindrome.  In
particular, the initial zero run and the terminal one run both have length
\(\varrho\).

For \(M\geq1\), put
\[
 \mathcal J_M=(1,3,5,\ldots,2\lfloor(M-1)/2\rfloor-1).
\]
This word is empty for \(M=1,2\).

\Needspace{5\baselineskip}
Suspending the root word amounts to two applications of $T$ on the canonical word.
\begin{lemma}
\label{lem:root-suspension}
Let \(U\) be the canonical word of a nonempty root word \(R\).  Then the
canonical word of
\[
 \Gamma(R)=(1,R+2)
\]
is
\[
 \Lam(U)=T^2(U)[1:|T^2(U)|-1).
\]
The map \(\Lam\) is injective and commutes with \(\Star\).
\end{lemma}

\begin{proof}
Write \(N=|U|/2\), let \(r=|R|\), and put \(B_T=T(U)\).  If
\(\beta\) is the degree word of \(B_T\), Lemma~\ref{lem:zero-data} gives
\[
 \beta_1=0,
 \qquad
 \beta_{s+1}=
 \begin{cases}
 0,&U_s=0,\\
 c_k+1,&U_s\text{ is the }k\text{-th one}.
 \end{cases}
\]
The sum of these degrees is
\[
 \sum_{k=1}^N(c_k+1)=(N-r)+N=2N-r.
\]
The word \(B_T\) has half-length \(2N+1\), so it has \(r+1\) roots.  Its
last \(r+1\) degrees are \((0,R+1)\).  To see the suffix directly, the last
zero of \(U\) is followed by its \(r\) terminal ones, whose one-ranks are
\(N-r+1,\ldots,N\).  These are precisely the root indices of \(U\), so the
formula for \(\beta\) gives the terminal degree zero followed by the root
degrees increased by one.

Apply \(\mathcal C\) once more.  Since a canonical word satisfies
\(\mathsf S\), Lemma~\ref{lem:profile-toggle} gives the exact implication
chain
\[
 \mathsf S(U)\Longrightarrow\mathsf R(B_T)
 \Longrightarrow\mathsf S(\mathcal C(B_T)).
\]
Thus \(\mathcal C(B_T)=\Lam(U)\) satisfies \(\mathsf S\).  Its number of roots is
still \(r+1\).  Indeed, if \(N_B=2N+1\) is the half-length of \(B_T\), and
\(\gamma\) is the degree word of \(\mathcal C(B_T)\), then
\[
 \sum\gamma=\sum_{k=1}^{N_B}(\beta_k+1)=2N_B-(r+1).
\]
The new core has half-length \(2N_B\), so it has exactly \(r+1\) roots.
The same terminal ones of \(B_T\), now indexed by the root degrees
\((0,R+1)\), produce its last \(r+1\) degrees after the second application:
\[
 (0,R+1)+1=(1,R+2).
\]
Uniqueness in Lemma~\ref{lem:root-grammar} identifies this word with the
canonical word of \(\Gamma(R)\).

Put \(V=\Lam(U)\).  The lengths recover the two preceding words exactly:
\[
 |B_T|=\frac{|V|}{2},
 \qquad
 |U|=\frac{|B_T|}{2}-1.
\]
With bits indexed from zero and degrees from one, the gap formula gives
\[
 (B_T)_t=1\quad\Longleftrightarrow\quad\gamma_{t+1}>0
 \qquad(0\leq t<|B_T|).
\]
The degree word \(\beta\) of the recovered word \(B_T=T(U)\) then gives
\[
 U_t=1\quad\Longleftrightarrow\quad\beta_{t+2}>0
 \qquad(0\leq t<|U|).
\]
Thus \(U\) is recovered letter by letter, and \(\Lam\) is injective.
Commutation with \(\Star\) follows from Lemma~\ref{lem:T-star}.
\end{proof}

\Needspace{5\baselineskip}
The anti-palindromic words satisfying $\mathsf S$ are completely classified by their roots.
\begin{theorem}
\label{thm:root-classification}
The roots of the nonempty anti-palindromic Dyck words satisfying
\(\mathsf S\) are exactly
\begin{equation*}
 R=0^b\qquad(b\geq1)
\end{equation*}
and
\begin{equation}
\label{eq:classified-root-positive}
 R=\mathcal J_M(M-1)^aM^b
 \qquad(M\geq1,\ a\geq0,\ b\geq1).
\end{equation}
\end{theorem}

\begin{proof}
The cases \(M=0,1,2\) follow directly from the run vector.  For
\(M=0\), one has \(R=0^b\).  Suppose \(M=1\), let \(b\) be the number of
letters one in \(R\), and let \(z\) be the number of its initial zeros.
Since \(c=0^bR\), the initial zero run has length \(b+z\), whereas the
terminal one run has length \(|R|\).  Their equality forces every zero of
\(R\) to occur before every one.  Thus \(R=0^a1^b\).

Suppose \(M=2\).  Let \(a\) count the letters one before the first letter two
of \(R\), and let \(b\) count all letters two.  The identities
\[
 \varphi(1)=(0),
 \qquad \varphi(2)=(1,0),
 \qquad \varphi^2(2)=(0)
\]
show that the initial zero run has length \(a+b\).  Equality with the
terminal run \(|R|\) excludes both a zero anywhere in \(R\) and a one after
the first two.  Hence \(R=1^a2^b\).

Assume \(M\geq3\).  Equation \eqref{eq:phi-binomial-length} shows that
\(\varphi^{M-1}(M)\) has length \(M\) and contains only zeros and ones.  For
every word \(w\) of nonnegative integers,
\[
 |\varphi(w)|=\sum_i w_i.
\]
Since \(|\varphi^M(M)|=1\), the sum of the letters of
\(\varphi^{M-1}(M)\) is one.  Thus this word contains one letter one and
\(M-1\) zeros.

Let \(h_M\) be the number of zeros before that one.  In \(\varphi(M)\), after
another \(M-2\) iterations, only the letter \(M-1\) contributes the unique
one.  The letter \(M-2\) contributes one zero, and every smaller letter has
disappeared.  The explicit definition of \(\sigma_M\) places \(M-2\) before
\(M-1\) in \(\varphi(M)\) exactly when \(M\) is odd.  Hence
\[
 h_1=0,
 \qquad
 h_M=h_{M-1}+\mathbf 1_{\{M\text{ is odd}\}}.
\]
Therefore \(h_M=\lfloor(M-1)/2\rfloor\), and
\begin{equation*}
 \varphi^M(M)=(0),
 \qquad
 \varphi^{M-1}(M)=0^t1\,0^{M-t-1},
 \qquad t=\left\lfloor\frac{M-1}{2}\right\rfloor.
\end{equation*}
Let \(b\) be the number of root letters \(M\), and let \(a\) be the number
of root letters \(M-1\) before the first \(M\).  The first zero run in the
degree word has length \(a+b+t\).  Anti-palindromicity gives
\begin{equation*}
 \varrho=|R|=a+b+t.
\end{equation*}

The positive nonroot degrees cannot end inside this initial comparison.  Let
\(\nu_x\) be the number of positive-degree nonroot descendants below a root of
type \(x\).  The root has one child of each type from zero through \(x-1\),
so
\[
 \nu_1=0,
 \qquad
 \nu_x=(x-1)+\sum_{y=1}^{x-1}\nu_y.
\]
Induction gives \(\nu_x=2^{x-1}-1\).  Let \(N_+(w)\) denote the number of
positive letters of an integer word \(w\).  Consequently,
\[
 N_+(\varphi(c))\geq b\nu_M+a\nu_{M-1}.
\]
Subtracting \(\varrho=a+b+t\) yields
\begin{equation*}
 N_+(\varphi(c))-\varrho
 \geq b(2^{M-1}-2)+a(2^{M-2}-2)-t\geq1.
\end{equation*}
For \(M=3\), the last expression is at least \(2b-1\).  For \(M\geq4\), the
term containing \(b\) suffices.  This is a lower bound, not an equality.
In particular, the first \(\varrho+1\) positive letters of \(c\) lie in the
prefix \(\varphi(c)=c_1\cdots c_{n-\varrho}\), strictly before \(R\).

Write
\[
 v_i=c_{j_i}>0,
 \qquad
 z_i=j_{i+1}-j_i\quad(1\leq i<s),
 \qquad
 z_s=n-j_s+1.
\]
The word \(c\) begins with \(0^{\varrho}\).  In the factorization
\(c=\varphi(c)R\), only a factor \(\varphi(1)=(0)\) can produce this initial
zero run, because \(\varphi(L)\) begins with one for \(L\geq2\).  There are
therefore exactly \(\varrho\) source letters one before the first source
letter at least two.  The preceding margin places their positive outputs
strictly before the root suffix.  Consequently,
\[
 v_1=\cdots=v_{\varrho}=1,
 \qquad v_{\varrho+1}\geq2.
\]
For \(i<\varrho\), the output \(v_i=1\) must be the unique positive output of
\(\varphi(2)=(1,0)\).  Otherwise it would be the first output of a factor
\(\varphi(L)\) with \(L\geq3\), whose next positive output is at least two,
contrary to \(v_{i+1}=1\).  Hence \(z_i\geq2\).  For \(i=\varrho\), the
factor cannot be \(\varphi(2)\), since the next positive output would again
begin with one, contrary to \(v_{\varrho+1}\geq2\).  It begins a factor
\(\varphi(L)\) with \(L\geq3\), whose first two positive outputs are adjacent.
Thus \(z_{\varrho}=1\).

Palindromicity of \eqref{eq:run-vector} gives, for
\(1\leq i\leq\varrho\), the two families of equalities
\[
 z_{s+1-i}=v_i=1,
 \qquad
 v_{s+1-i}=z_i.
\]
The first family says
\(j_s=n,j_{s-1}=n-1,\ldots,j_{s+1-\varrho}=n+1-\varrho\).  Thus the last
\(\varrho\) degree positions are positive and form the root word.  The
second family gives
\[
 R_{\varrho+1-i}=z_i
 \qquad(1\leq i\leq\varrho).
\]
Consequently,
\begin{equation}
\label{eq:root-descent}
 R=(1,R'+2)
\end{equation}
for a root word \(R'\) of maximum \(M-2\).  It is nonempty, because
\(M\geq3\) gives \(t\geq1\), and hence
\(\varrho=a+b+t\geq b+t\geq2\).

Let \(U\) be the canonical word of \(R'\).  Lemma
\ref{lem:root-suspension} identifies the canonical word of \(R\) with
\(\Lam(U)\).  Since this word is fixed by \(\Star\), injectivity and
commutation give \(U=U^{\Star}\).  Thus \(R'\) is again an anti-\(\mathsf S\)
root.  Iterating \eqref{eq:root-descent} reaches a base of maximum zero, one,
or two.  The identity
\[
 \mathcal J_M=(1,\mathcal J_{M-2}+2)
\]
then gives \eqref{eq:classified-root-positive}.

Conversely, the three base families have canonical words
\[
 0^b1^b,
 \qquad
 0^{a+b}(10)^b1^{a+b},
\]
and
\[
 0^{a+b}(100)^b(10)^a(110)^b1^{a+b}.
\]
Their outer runs are interchanged by \(\Star\), while
\[
 ((10)^b)^{\Star}=(10)^b
\]
and
\[
 \bigl((100)^b(10)^a(110)^b\bigr)^{\Star}
 =(100)^b(10)^a(110)^b.
\]
Thus the extreme runs \(0^{a+b}\) and \(1^{a+b}\) correspond exactly, the
three words satisfy \(\mathsf S\), and they are fixed by \(\Star\).  Lemma
\ref{lem:root-suspension} preserves both properties.  This proves the
converse and completes the classification.
\end{proof}

\Needspace{5\baselineskip}
The orbital root words are the initial odd segments.
\begin{corollary}
\label{cor:orbital-roots}
The root word of \(A_r\) is
\begin{equation*}
 R_r=(1,3,5,\ldots,2r+1).
\end{equation*}
It has \(\varrho_r=r+1\) letters.
\end{corollary}

\begin{proof}
The root of \(A_0=0101\) is \((1)\).  Lemma
\ref{lem:root-suspension} sends each root to \((1,R+2)\).
\end{proof}

\section{Toeplitz forests}

\subsection{Reflection}

For an anti-palindromic Dyck word of length \(2n\), define
\[
 \mathcal R(i)=n+1-i.
\]

\Needspace{5\baselineskip}
Anti-palindromicity makes the zero data self-conjugate.
\begin{lemma}
Let \(W\) be an anti-palindromic Dyck word of length \(2n\), with zero
coordinates \(p_m,q_m,c_m,d_m\).
For \(0\leq r\leq n\),
\begin{equation}
\label{eq:self-conjugate}
 q_r=\#\{1\leq i\leq n:q_i\geq n+1-r\}.
\end{equation}
Consequently, with
\[
 \mathcal M_{i,t}=\mathbf1_{\{t\leq q_i\}},
\]
one has
\begin{equation}
\label{eq:incidence-centro}
 \mathcal M_{i,t}=\mathcal M_{\mathcal R(t),\mathcal R(i)}
 \qquad(1\leq i,t\leq n).
\end{equation}
\end{lemma}

\begin{proof}
For \(r=0\), one has \(q_0=0\), and the set on the right of
\eqref{eq:self-conjugate} is empty because \(q_i\leq n-1\).  Assume now that
\(1\leq r\leq n\).
The inequality \(q_i\geq n+1-r\) means that the \((n+1-r)\)-th one occurs
before the \(i\)-th zero.  Reflection and complementation send this ordered
pair to the \((n+1-i)\)-th one before the \(r\)-th zero.  As \(i\) varies,
these are precisely the \(q_r\) ones preceding that zero.  This proves
\eqref{eq:self-conjugate}.

Since \(q_i\) is nondecreasing,
\[
 q_{\mathcal R(t)}=\#\{j:q_j\geq t\}.
\]
The right side is at least \(\mathcal R(i)\) exactly when \(q_i\geq t\).
This is \eqref{eq:incidence-centro}.
\end{proof}

Extend the reflection by
\[
 \mathcal R(0)=n+1,
 \qquad \mathcal R(n+1)=0,
\]
and put \(\mathcal M_{0,t}=\mathcal M_{i,n+1}=0\).  These are the boundary values used below.

An edge \(t\to m\) is \emph{active} if
\[
 q_{m-1}<t<q_m.
\]
Add the virtual active edges \(t\to n+1\) for \(q_n<t\leq n\).  The child
\(t=q_m\) is terminal and is not active.

\Needspace{5\baselineskip}
The reflection exchanges active edges and leaves.
\begin{lemma}
\label{lem:edge-leaf}
Let \(W\) be an anti-palindromic Dyck word of length \(2n\), with the
reflection and active edges defined above.
The reflection \(t\mapsto r=n+1-t\) is a bijection from active edges,
including virtual edges, to leaves \(r\) with \(c_r=0\).  If the parent is
\(m\), where \(m=n+1\) for a virtual edge, then
\begin{equation*}
 d_r=m-t,
 \qquad
 \Delta_r=\omega_t,
 \qquad
 \Delta_r:=p_{n+1-r}-p_{n-r}.
\end{equation*}
\end{lemma}

\begin{proof}
For a real active edge, \(m\) is the first index for which \(q_m>t\).  No
value \(q_i\) equals \(t\).  Equation \eqref{eq:self-conjugate} gives
\[
 q_r=\#\{i:q_i\geq t\}=n-m+1.
\]
Hence \(d_r=m-t\).  Differencing \eqref{eq:self-conjugate} shows \(c_r=0\),
and reflection gives \(\Delta_r=\omega_t\).  For \(t>q_n\), one has \(q_r=0\),
which gives the same formulas with parent \(n+1\).

Conversely, if \(c_r=0\), then \(t=n+1-r\) is not a value of the sequence
\(q\).  It either lies above \(q_n\) or strictly inside the unique jump that
crosses it.  This recovers one active edge.
\end{proof}

For the remainder of this section, assume in addition that the word satisfies
\(\mathsf S\).  The terminal child \(q_j\) in a source factor
of positive degree has degree \(\varphi(c_j)_{c_j}=0\).  Hence \(c_m>0\)
implies that \(m\) is not a value of the sequence \(q\).  Lemma
\ref{lem:edge-leaf}, applied to the unique active edge whose child is \(m\),
gives
\begin{equation}
\label{eq:positive-degree-mirror}
 c_m>0\Longrightarrow c_{n+1-m}=0,
 \qquad
 \omega_m>1\Longrightarrow\Delta_m=1.
\end{equation}

\begin{definition}
The canonical forest is \emph{Toeplitz} if two active edges of the same span
\(h=m-t\) have the same label \(\omega_t\).  This common label is denoted by
\(\alpha_h\).
\end{definition}

Let \(D=\max_m d_m\).  The first passage of \(d\) at each height
\(1\leq h\leq D\) is a leaf, so Lemma~\ref{lem:edge-leaf} supplies an active
edge of every span in this range.  If \(c_m=L>0\), property \(\mathsf S\)
gives \(\omega_{q_{m-1}+j}=\sigma_L(j)\), and its \(j\)-th child is
\(t=q_{m-1}+j\).  The active children are those with \(1\leq j<L\), and
\[
 m-t=d_{m-1}+1-j.
\]
Toeplitz is therefore equivalent to
\begin{equation}
\label{eq:toeplitz-leaf}
 c_m=0\Longrightarrow\Delta_m=\alpha_{d_m}
\end{equation}
and
\begin{equation}
\label{eq:toeplitz-node}
 c_m=L>0\Longrightarrow
 (\alpha_{d_m+1},\ldots,\alpha_{d_{m-1}})
 =\rev(\sigma_L(1),\ldots,\sigma_L(L-1)).
\end{equation}
For \(L=1\), both words in \eqref{eq:toeplitz-node} are empty.

\subsection{Suspension}

Read a canonical forest in breadth-first order from roots to leaves.  If a
vertex has degree \(L\), its children have degrees
\[
 \mathscr P_L=\rev(\varphi(L)).
\]
The explicit definition of \(\sigma_L\) gives
\begin{equation}
\label{eq:gadget-identity}
 \mathscr P_{L+2}=(0,\mathscr P_L+2,1).
\end{equation}

Number the old vertices in breadth-first order by \(u=1,\ldots,n\), and set
\[
 \check c_u=c_{n+1-u}.
\]
Under \(R\mapsto\Gamma(R)=(1,R+2)\), add a separate root of degree one.  An
old vertex \(u\) of degree \(\check c_u\) becomes a core \(\operatorname{core}(u)\) of degree
\(\check c_u+2\).  If \(v_1,\ldots,v_{\check c_u}\) are its old children, its new children
are
\begin{equation*}
 x_u,\operatorname{core}(v_1),\ldots,\operatorname{core}(v_{\check c_u}),y_u.
\end{equation*}
The vertices \(x_u\) and the child \(\zeta_u\) of \(y_u\) have degree zero,
while \(y_u\) has degree one.

Let \(K_u\) be the difference between the breadth-first ranks of \(x_u\) and
\(\operatorname{core}(u)\).  The active edges out of \(\operatorname{core}(u)\) have consecutive spans
\[
 K_u+1,\ldots,K_u+\check c_u+1.
\]

\Needspace{5\baselineskip}
Consecutive core ranks differ by an explicit binomial correction.
\begin{lemma}
\label{lem:separators}
Let \(W\) be a nonempty anti-palindromic canonical word of length \(2n\), and apply the
suspension \(R\mapsto(1,R+2)\) described above.  If the old root word has
length \(\varrho\), then
\begin{equation*}
 K_1=\varrho+1
\end{equation*}
and, for \(1\leq u<n\),
\begin{equation}
\label{eq:separator-step}
 K_{u+1}-K_u
 =\check c_u+1-\binom{c_{u+1}+1}{2}.
\end{equation}
\end{lemma}

The ordered proof, including the root boundary, child separators,
grandchild separators, and the case \(n=1\), is given in
Appendix~\ref{app:separators}.
Anti-palindromicity is used there to identify reversed depth order with the
breadth-first order.  Canonicity alone does not provide that identification.

Suppose the old forest is Toeplitz.  Put
\begin{equation*}
 H_0=\varrho+1,
 \qquad H_h=\varrho+1+\sum_{s=1}^h\alpha_s.
\end{equation*}

\Needspace{5\baselineskip}
In a Toeplitz forest, the labels locate every core in the breadth-first queue.
\begin{lemma}
\label{lem:core-queue}
Assume that the old forest is Toeplitz, with labels \(\alpha_h\) and numbers
\(H_h\) defined above.  For \(1\leq u\leq n\),
\begin{equation}
\label{eq:core-queue-position}
 K_u=H_{d_u}-\check c_u-1.
\end{equation}
\end{lemma}

\begin{proof}
The Dyck condition gives \(c_1=0\) and \(d_1=1\).  The leaf \(r=1\) is dual
to the virtual edge \(n\to n+1\).  Its label is
\(\omega_n=\check c_1+1\), so \(\alpha_1=\check c_1+1\).  This gives
\(H_1-\check c_1-1=\varrho+1=K_1\).

Assume \eqref{eq:core-queue-position} at \(u\), and put \(L=c_{u+1}\).
Since \(d_{u+1}=d_u+1-L\), there are two cases.  If \(L=0\), edge-leaf
duality gives \(\check c_{u+1}+1=\alpha_{d_{u+1}}\), and the right side of
\eqref{eq:core-queue-position} increases by \(\check c_u+1\).  This equals
\eqref{eq:separator-step}.

If \(L>0\), equation \eqref{eq:positive-degree-mirror} gives
\(\check c_{u+1}=0\), while \eqref{eq:toeplitz-node} gives
\[
 \sum_{h=d_{u+1}+1}^{d_u}\alpha_h
 =\sum_{j=1}^{L-1}\sigma_L(j)
 =\binom{L+1}{2}-1.
\]
The change is again \(\check c_u+1-\binom{L+1}{2}\).  Lemma
\ref{lem:separators} completes the induction.  Empty sums include \(L=1\).
\end{proof}

For \(L\geq1\), define
\[
 \mathcal A_L=\rev(\sigma_L(1),\ldots,\sigma_L(L-1)).
\]

\Needspace{5\baselineskip}
Suspension preserves the Toeplitz property and rewrites the profile explicitly.
\begin{theorem}
\label{thm:toeplitz-suspension}
If a canonical anti-\(\mathsf S\) forest of root word
\(R=(R_1,\ldots,R_{\varrho})\) is Toeplitz, then the forest of root word
\(\Gamma(R)=(1,R+2)\) is Toeplitz.  Its ordered profile is
\begin{equation}
\label{eq:toeplitz-suspension}
 \alpha^{\Gamma}
 =(R_{\varrho}+3,R_{\varrho-1}+3,\ldots,R_1+3,2)
 \prod_{h=1}^{D}\mathcal A_{\alpha_h+1}.
\end{equation}
The product denotes concatenation in increasing order of \(h\).
\end{theorem}

\begin{proof}
The virtual edges of the \(\varrho\) root cores, followed by the separate
degree-one root, have spans \(1,\ldots,\varrho+1\) and labels
\[
 (R_{\varrho}+3,\ldots,R_1+3,2).
\]

Fix an old vertex \(u\), and put \(h=d_u\).  If \(c_u=0\), then
\(\alpha_h=\check c_u+1\) by edge-leaf duality.  Lemma~\ref{lem:core-queue} gives
\(K_u=H_{h-1}\).  The \(\check c_u+1\) active edges from \(\operatorname{core}(u)\) occupy exactly
the spans from \(H_{h-1}+1\) through \(H_h\).  Identity
\eqref{eq:gadget-identity} gives their labels
\[
 \mathcal A_{\alpha_h+1}.
\]
If \(c_u>0\), equation \eqref{eq:positive-degree-mirror} gives \(\check c_u=0\).
The single active edge to \(y_u\) has span \(H_h\) and label two, the last
entry of the same word \(\mathcal A_{\alpha_h+1}\).

Because \(d_u-d_{u-1}=1-c_u\leq1\), the first vertex at every height \(h\)
is a leaf and fills the whole span interval
\([H_{h-1}+1,H_h+1)\).  Later vertices at that height repeat the same labels
at the same spans.

The list is exhaustive.  The only active edges in the suspended forest are
the virtual edges and the edges from cores to children occurring after
\(x_u\).  The unique edge of the separate degree-one root and every edge
\(y_u\to \zeta_u\) is terminal and therefore inactive, while \(x_u\) and \(\zeta_u\)
are leaves.  The consecutive span intervals end at
\[
 H_D=\varrho+1+\sum_{h=1}^D\alpha_h.
\]
Lemma~\ref{lem:root-suspension} shows that the suspended forest is again
anti-\(\mathsf S\).  Edge-leaf duality therefore says that every new depth is
the span of an active edge.  No depth can exceed \(H_D\), and the first leaf
at depth \(D\) supplies span \(H_D\).  Thus
\[
 D^{\Gamma}=H_D.
\]
Concatenating the displayed blocks proves Toeplitz and
\eqref{eq:toeplitz-suspension} without assuming the new Toeplitz property
during the argument.
\end{proof}

\Needspace{5\baselineskip}
Every forest in the classification is Toeplitz.
\begin{theorem}
Every canonical forest associated with an anti-palindromic Dyck word
satisfying \(\mathsf S\) is Toeplitz.
\end{theorem}

\begin{proof}
The three base root families in Theorem~\ref{thm:root-classification} have
profiles
\begin{equation}
\label{eq:base-toeplitz-profiles}
 0^b\mapsto1^b,
 \qquad
 0^a1^b\mapsto2^b1^a,
 \qquad
 1^a2^b\mapsto3^b2^{a+b}.
\end{equation}
Appendix~\ref{app:toeplitz-bases} checks every active edge in these families.
The classification descends every anti-\(\mathsf S\) root to one of them.
Theorem~\ref{thm:toeplitz-suspension} carries Toeplitz through every upward
suspension.
\end{proof}

\subsection{The minimum law}

For \(0\leq m\leq n\), set
\begin{equation*}
 b_m=p_{n-m}+p_{q_m}.
\end{equation*}
For \(1\leq m\leq n\), set
\[
 \Delta_m=p_{n+1-m}-p_{n-m}=\omega_{n+1-m},
 \qquad
 \underline b_{m-1}=\min_{0\leq r<m}b_r.
\]
For a finite real sequence \(x\) on a domain containing zero, define its
first-passage set by
\[
 \FP(x)=\{0\}\cup
 \{t>0:x(t)>x(s)\text{ for every }0\leq s<t\}.
\]

\Needspace{5\baselineskip}
The potential is constrained by the following minimum law.
\begin{definition}
The word satisfies the \(b\)-minimum property (BM) if, for
\(1\leq m\leq n\),
\begin{equation}
\label{eq:BM}
 \begin{cases}
 b_m=\underline b_{m-1}-\Delta_m,&m\in\FP(d),\\
 b_m\geq\underline b_{m-1},&m\notin\FP(d).
 \end{cases}
\end{equation}
\end{definition}

Under \(\mathsf S\), the sum of the source gaps at degree \(L=c_m\) is
\[
 p_{q_m}-p_{q_{m-1}}
 =\sum_{j=1}^L\sigma_L(j)
 =\binom{L+1}{2}.
\]
Consequently,
\begin{equation}
\label{eq:b-increment}
 b_m-b_{m-1}=\binom{c_m+1}{2}-\Delta_m.
\end{equation}

\Needspace{5\baselineskip}
The potential has an exact expression in the labels.
\begin{lemma}

For every anti-palindromic Dyck word satisfying \(\mathsf S\),
\begin{equation}
\label{eq:potential}
 b_m=b_0-\sum_{h=1}^{d_m}\alpha_h
 \qquad(0\leq m\leq n).
\end{equation}
\end{lemma}

\begin{proof}
For \(m=0\), both sides of \eqref{eq:potential} equal \(b_0\).  If \(c_m=0\), then
\(d_m=d_{m-1}+1\).  Equations \eqref{eq:toeplitz-leaf} and
\eqref{eq:b-increment} give the change \(-\alpha_{d_m}\).

If \(c_m=L>0\), equation \eqref{eq:positive-degree-mirror} gives
\(\Delta_m=1\), and \eqref{eq:toeplitz-node} gives
\[
 b_m-b_{m-1}
 =\binom{L+1}{2}-1
 =\sum_{h=d_m+1}^{d_{m-1}}\alpha_h.
\]
This is exactly the change on the right side of \eqref{eq:potential}.  The
case \(L=1\) has an empty sum.
\end{proof}

\Needspace{5\baselineskip}
The minimum law holds throughout the classification.
\begin{theorem}
\label{thm:ASBM}
Every anti-palindromic Dyck word satisfying \(\mathsf S\) satisfies BM.
\end{theorem}

\begin{proof}
Every Toeplitz label is a positive gap.  The right side of
\eqref{eq:potential} is therefore a strictly decreasing function of \(d_m\).
The minimum among \(b_0,\ldots,b_{m-1}\) occurs exactly at the indices where
\(d\) has reached its current maximum.  A new record of \(d\) is a step with
\(c_m=0\) and raises that maximum by one.  Equation
\eqref{eq:toeplitz-leaf} then gives the first clause of \eqref{eq:BM}.  At a
nonrecord, \(d_m\) does not exceed the previous maximum, and
\eqref{eq:potential} gives the second clause.

At \(m=1\), the Dyck condition gives \(c_1=0\), so \(d_1=1\) is a record.  At
\(m=n\), the quantity \(\Delta_n=p_1-p_0=\omega_1\) is defined.  If \(q_m=0\),
then \(c_1=\cdots=c_m=0\), so \(m\) is again a record.  These observations
cover the initial, terminal, and empty-factor boundaries.
\end{proof}

The records determine how the maximum \(V_r\), the last first-passage height
of \(P_r\), is transported.  Section~\ref{sec:records} tracks every first
passage, including the maximum, under \(T\) and \(T^2\).  These images order
the fibers in Section~\ref{sec:fibers}.

\section{Records}
\label{sec:records}

\subsection{Record synchronization}

For a word \(W\), write \(\Rec(h_W)=\FP(h_W)\).  Every positive record cut of
a Dyck word occurs immediately after a zero.

For a Dyck word \(W\) of length \(2n\), define
\begin{equation}
\label{eq:pi-E}
 \pi(0)=0,
 \qquad \pi(k)=p_{O_W(k-1)}\quad(1\leq k\leq2n+1),
 \qquad E(k)=k-\pi(\pi(k)).
\end{equation}

\Needspace{5\baselineskip}
The synchronization property couples the first passages of $E$ with the records of the height.
\begin{definition}
The word \(W\) satisfies record synchronization (RS) if
\begin{equation*}
 k\in\FP(E)
 \quad\Longleftrightarrow\quad
 \pi(k)\in\Rec(h_W)
 \qquad(1\leq k\leq2n+1).
\end{equation*}
\end{definition}

\Needspace{5\baselineskip}
The minimum law is equivalent to record synchronization.
\begin{theorem}
\label{thm:BM-RS}
For every anti-palindromic Dyck word, BM is equivalent to RS on the exact
domain \(1\leq k\leq|W|+1\).
\end{theorem}

\begin{proof}
Use the one cuts \(s_m\) and the terminal sentinels from
Definition~\ref{def:zero-coordinates}.
The fiber of \(\pi\) above \(p_m\), restricted to the RS domain, is
\begin{equation*}
 \mathcal K_m=\{s_m+1,\ldots,s_{m+1}\}.
\end{equation*}
Anti-palindromicity gives
\[
 s_m=2n+1-p_{n+1-m}.
\]
On \(\mathcal K_m\), one has \(\pi^2(k)=p_{q_m}\), so \(E\) is an increasing
ramp.  Its first and last values are
\begin{equation}
\label{eq:E-ramp-ends}
 2n+2-b_m-\Delta_m,
 \qquad 2n+1-b_m.
\end{equation}
Before this ramp, the largest value of \(E\) is
\begin{equation}
\label{eq:E-previous-max}
 2n+1-\underline b_{m-1}.
\end{equation}
The function \(\pi\) is nondecreasing, and hence
\[
 E(k+1)-E(k)\leq1.
\]

The cut \(p_m\) is a height record exactly when \(m\in\FP(d)\).  If it is a
record, RS requires every point of the ramp \(\mathcal K_m\) to be a first
passage of \(E\).  Its first value must therefore equal the previous maximum
plus one.  Equations \eqref{eq:E-ramp-ends} and
\eqref{eq:E-previous-max} give
\[
 b_m=\underline b_{m-1}-\Delta_m.
\]
If \(p_m\) is not a record, RS requires no point of the ramp to be a first
passage.  Its last value must not exceed the previous maximum, which gives
\[
 b_m\geq\underline b_{m-1}.
\]
These are the two BM clauses, and every implication is reversible.

For \(m=0\), one has \(\pi(k)=0\) and \(E(k)=k\) on the first ramp, so RS is
automatic.  The terminal sentinel of
Definition~\ref{def:zero-coordinates} includes the last domain point in
\(\mathcal K_n\).
\end{proof}

\Needspace{5\baselineskip}
Every orbital core synchronizes its records.
\begin{corollary}
\label{cor:orbital-RS}
Every orbital core \(A_r\) satisfies RS.
\end{corollary}

\begin{proof}
Corollaries~\ref{cor:orbital-S} and \ref{cor:orbital-anti} place \(A_r\) in
the domain of Theorem~\ref{thm:ASBM}.  Theorems~\ref{thm:ASBM} and
\ref{thm:BM-RS} give RS.
\end{proof}

\subsection{Transport under \texorpdfstring{\(T\)}{T} and \texorpdfstring{\(T^2\)}{T2}}

\Needspace{5\baselineskip}
The functions $\theta_W$ and $e_W$ have the following elementary properties.
\begin{lemma}
\label{lem:theta-e}
For a Dyck word \(W\), the function \(\theta_W\) is nondecreasing and
\(\theta_W(t)<t\) for \(t\geq1\).  Moreover,
\begin{equation}
\label{eq:e-height}
 e_W(t)=h_W(t)+d_{O_W(t)}.
\end{equation}
The increment of \(e_W\) is one at a zero of \(W\) and is nonpositive at a
one.
\end{lemma}

\begin{proof}
The function \(O_W\) is nondecreasing and the zero cuts are increasing.  The
Dyck inequality gives \(\theta_W(t)\leq t\).  Equality at positive \(t\)
would force the preceding height to be \(-1\), as in Lemma
\ref{lem:T-dyck}.  If \(m=O_W(t)\), then
\[
 e_W(t)=t-p_m=Z_W(t)-q_m=h_W(t)+m-q_m.
\]
This is \eqref{eq:e-height}.  The increment statement follows directly from
the definition of \(\theta_W\).
\end{proof}

\Needspace{5\baselineskip}
The records of an image under $T$ have an explicit description.
\begin{theorem}
\label{thm:T-records}
Let \(W\) be a Dyck word of length \(2n\), let \(Z=T(W)\), and let
\(p'_k,d'_k\) be its zero data, with \(p'_0=q'_0=d'_0=0\).  Then
\begin{equation}
\label{eq:T-record-set}
 \Rec(h_Z)=\{0\}\cup
 \{p'_k:1\leq k\leq2n+1,\ k-1\in\FP(e_W)\},
 \qquad d'_k=e_W(k-1)+1
\end{equation}
for \(1\leq k\leq2n+1\).
With \(\pi(k)=\theta_W(k-1)\),
\begin{equation}
\label{eq:theta-parent-transport}
 \theta_Z(p'_k)=p'_{\pi(k)}
 \qquad(1\leq k\leq2n+1).
\end{equation}
If
\[
 \pi(0)=E(0)=0,
 \qquad E(k)=e_Z(p'_k)\quad(1\leq k\leq2n+1),
\]
then
\begin{equation}
\label{eq:E-id-pi2}
 E(k)=k-\pi^2(k),
\end{equation}
and
\begin{equation}
\label{eq:FP-eZ}
 \FP(e_Z)=\{0\}\cup
 \{p'_k:k\in\FP(E),\ 1\leq k\leq2n+1\}.
\end{equation}
For every \(1\leq k\leq2n+1\), one also has the four-head formula
\begin{equation}
\label{eq:four-head}
 E(k)=1+h_W(k-1)+2h_W(\pi(k))+h_W(\pi^2(k)).
\end{equation}
When \(\pi(k)=0\), its last two height terms vanish and the formula reduces
to \(E(k)=k=1+h_W(k-1)\).
\end{theorem}

\begin{proof}
Lemma~\ref{lem:zero-data} gives
\[
 d'_k=k-q'_k=e_W(k-1)+1.
\]
Between two consecutive zero cuts, a binary height first decreases through
the intervening ones and then increases by one.  All strict records therefore
occur at zero cuts.  This proves \eqref{eq:T-record-set}.

At \(p'_k\), the number of ones already emitted is \(q'_k=\pi(k)\).  The
definition of \(\theta_Z\) gives \eqref{eq:theta-parent-transport}.  Since
\(p'_m=m+\pi(m)\),
\[
 E(k)=p'_k-p'_{\pi(k)}=k-\pi^2(k),
\]
which is \eqref{eq:E-id-pi2}.  Lemma~\ref{lem:theta-e} applied to \(Z\)
shows that first passages of \(e_Z\) occur at zero cuts.  Restriction to those
cuts gives \eqref{eq:FP-eZ}.

The height-addition identity at the zero cut \(p'_k\) gives
\[
 d'_k=1+h_W(k-1)+h_W(\pi(k)).
\]
The convention \(d'_0=0\) in
Definition~\ref{def:zero-coordinates} makes
\(E(k)=d'_k+d'_{\pi(k)}\) valid also when \(\pi(k)=0\).  If
\(\pi(k)\geq1\), this cut follows a zero,
so \(h_W(\pi(k)-1)=h_W(\pi(k))-1\).  Substitution gives
\eqref{eq:four-head}.  If \(\pi(k)=0\), the initial zero run gives
\(E(k)=k=1+h_W(k-1)\).  Since \(\pi^2(k)=0\) and \(h_W(0)=0\), this is the
same displayed formula.
\end{proof}

\subsection{Projection, compression, and adjacency}

For a nonempty Dyck word \(X\), put
\[
 \mathcal D_\psi(X)=\{y\in\{0,\ldots,|X|\}:O_X(y)<Z(X)\}
\]
and define
\begin{equation}
\label{eq:compression-map}
 \psi_X(y)=p^X_{O_X(y)+1}.
\end{equation}
The word \(X\) satisfies record compression (RC) if \(\psi_X\) maps every record cut in its domain
to a record cut.

For a core \(A\), define \(\mathsf{Cl}_\theta(A)\) by
\begin{equation}
\label{eq:theta-closure}
 m\in\FP(d_A)\Longrightarrow q_m\in\FP(d_A)\cup\{0\}.
\end{equation}
If \(P=0A1\), then
\begin{equation}
\label{eq:phi-theta-shift}
 \psi_P(y)=\theta_A(y-1)+1
 \qquad(y\in\mathcal D_\psi(P),\ y\geq1),
\end{equation}
and therefore
\begin{equation}
\label{eq:RC-Cl-equivalence}
 RC(P)\quad\Longleftrightarrow\quad\mathsf{Cl}_\theta(A).
\end{equation}

Indeed, for every \(y\geq1\) in the displayed domain,
\[
 O_P(y)=O_A(y-1),
 \qquad
 p^P_{m+1}=p^A_m+1,
 \qquad
 h_P(y)=1+h_A(y-1).
\]
These identities prove \eqref{eq:phi-theta-shift}.  The boundary value
\(y=0\) is sent to \(\psi_P(0)=p^P_1=1\), which is a record cut.  If
\(y-1=p^A_m\) is any other record cut, then
\[
 \theta_A(p^A_m)=p^A_{q_m}.
\]
It is sent to a record cut, or to the initial cut, exactly when
\(q_m\in\FP(d_A)\cup\{0\}\).  This proves
\eqref{eq:RC-Cl-equivalence}, including both boundary cuts.

Let \(A\) be a nonempty Dyck word of length \(2n\), and fix the transition
\begin{equation}
\label{eq:transition-PBQ}
 P=0A1,
 \qquad B_T=T(A)=0N,
 \qquad Q=T(B_T)=T^2(A).
\end{equation}
Let \(\mathcal D_Q\) be the set of values \(O_Q(t)\) at record cuts \(t\) of
\(Q\).  The Law~1 head on its state-zero tape satisfies
\begin{equation}
\label{eq:fast-head}
 \operatorname{head}_0(t)=Z_N(t-1)\qquad(t\geq1).
\end{equation}
For \(\lambda\in\mathcal D_Q\), \(\lambda\geq2\), project \(\lambda\) to
\(\operatorname{head}_0(\lambda-1)+3\), and add the two initial values one
and two.  Record-parameter projection (RPP) is the assertion that this ordered projection is exactly the
set of positive record cuts of \(P\).

Put
\[
 V(P)=\max_{0\leq t\leq|P|}h_P(t),
 \qquad
 \rho_h=\min\{t:h_P(t)=h\}\quad(1\leq h\leq V(P)).
\]
Under RPP, denote the corresponding one-count parameter of \(Q\) by
\(\lambda_h\), and put
\[
 \chi_h=h_Q(p^Q_{\lambda_h+1}).
\]
The adjacency property (Adj) is the assertion
\begin{equation}
\label{eq:Adj}
 \chi_{h+1}=\chi_h+1
 \quad\Longleftrightarrow\quad
 O_P(\rho_{h+1})=O_P(\rho_h).
\end{equation}
This assertion is quantified by \(1\leq h<V(P)\).

\Needspace{5\baselineskip}
The three record statements are equivalent to the following conditions on the
maps $\pi$ and $E$.
\begin{lemma}
\label{lem:compressed-records}
Let \(A\) be a nonempty Dyck word of length \(2n\), form the transition
\eqref{eq:transition-PBQ}, and use \(\pi,E\) from \eqref{eq:pi-E} for \(A\).
Then
\begin{enumerate}[label=\textup{(\roman*)}]
\item
\[
 \mathcal D_Q=\{0\}\cup
 \{p'_{\pi(k)}(B_T):k\in\FP(E)\cap\{1,\ldots,2n+1\}\},
\]
and the corresponding one-count parameter is
\[
 \lambda(k)=\pi(k)+\pi^2(k).
\]
\item The word \(Q\) satisfies RC if and only if
\[
 k\in\FP(E)\Longrightarrow\pi(k)\in\FP(E)\cup\{0\}.
\]
\item For every resulting parameter \(\lambda\geq2\),
\[
 \operatorname{head}_0(\lambda-1)+3=\pi(k)+1.
\]
This projection is strictly increasing.  The parameters \(\lambda=0,1\) always belong
to \(\mathcal D_Q\) and supply the padding values one and two.
\item RPP is equivalent to
\begin{equation}
\label{eq:RPP-compressed}
 \pi\bigl(\FP(E)\cap\{0,\ldots,2n+1\}\bigr)=\Rec(h_A).
\end{equation}
\end{enumerate}
\end{lemma}

\begin{proof}
Apply Theorem~\ref{thm:T-records} first to \(A\) and then to \(B_T=T(A)\).
Equations \eqref{eq:theta-parent-transport} and \eqref{eq:FP-eZ} give the
one-count set and its closure criterion.  Since
\(p'_m=m+\pi(m)\), the parameter indexed by \(k\) is
\(\pi(k)+\pi^2(k)\).

In Law~1, the state-zero head moves once for every earlier zero of \(N\),
which proves \eqref{eq:fast-head}.  If \(\lambda=p'_m(B_T)\geq2\) follows a zero,
then
\[
 \operatorname{head}_0(\lambda-1)
 =Z_N(\lambda-2)=Z_{B_T}(\lambda)-2=m-2.
\]
Here \(m=\pi(k)\), so the projection is \(\pi(k)+1\).  Strict increase
follows from strict increase of the zero cuts.

Write the initial factor of \(A\) as \(0^r1\), where \(r\geq1\).  The zero
data formula gives the initial factors
\[
 B_T=0^{r+1}10\cdots,
 \qquad
 Q=0^{r+2}100\cdots.
\]
The initial zero run gives the one-count zero at record cuts of \(Q\), and
the second zero after the first one gives the one-count one.  These two
parameters correspond to the record cuts one and two of \(P\).  The remaining projected
values are \(\pi(k)+1\), while the positive records of \(P=0A1\) are one plus
the records of \(A\).  This proves \eqref{eq:RPP-compressed}.
\end{proof}

\Needspace{5\baselineskip}
Records are transported across the transition.
\begin{theorem}
\label{thm:RPP-RC-Adj}
Let \(A\) be a nonempty Dyck word of length \(2n\) satisfying RS and
\(\mathsf{Cl}_\theta(A)\).  For the
transition \eqref{eq:transition-PBQ}, RPP and Adj hold, and \(Q\) satisfies
RC.  More precisely,
\begin{equation}
\label{eq:Adj-gap}
 \chi_{h+1}-\chi_h\in\{1\}\cup[3,\infty).
\end{equation}
This formula holds for \(1\leq h<V(P)\).
\end{theorem}

\begin{proof}
RS gives both inclusions in \eqref{eq:RPP-compressed}.  For surjectivity,
take a record cut \(p_m\) of \(A\).  Use the one cuts and terminal sentinels
of Definition~\ref{def:zero-coordinates}.  Then
\[
 \pi^{-1}(p_m)\cap\{1,\ldots,2n+1\}
 =\{s_m+1,\ldots,s_{m+1}\}.
\]
This set is nonempty, including when \(m=n\), and contains a \(k\).  The reverse implication in RS gives
\(k\in\FP(E)\).  Lemma~\ref{lem:compressed-records} now proves RPP with its
order and padding.

If \(k\in\FP(E)\), RS gives \(\pi(k)\in\Rec(h_A)\).  If this value is
positive, write it as \(p_m\) with \(m\in\FP(d_A)\).  Then
\[
 \pi^2(k)=\theta_A(p_m-1)=p_{q_m}.
\]
The closure \eqref{eq:theta-closure} makes this another record or zero.  The
reverse implication in RS, applied to \(\pi(k)\), gives
\(\pi(k)\in\FP(E)\).  Lemma~\ref{lem:compressed-records} proves RC for \(Q\).

For adjacency, set
\[
 \widehat m_h=\rho_h-1.
\]
The RPP correspondence gives
\(\lambda_h=p'_{\widehat m_h}(B_T)\).  Therefore
\[
\begin{aligned}
 \chi_h
 &=p'_{\widehat m_h}+1-
   \theta_{B_T}(p'_{\widehat m_h})\\
 &=p'_{\widehat m_h}+1-p'_{\pi(\widehat m_h)}
  =E(\widehat m_h)+1.
\end{aligned}
\]
Thus the exact height identity is
\begin{equation*}
 \chi_h=E(\widehat m_h)+1.
\end{equation*}
If \(O_P(\rho_{h+1})=O_P(\rho_h)\), first-passage minimality forces
\(\rho_{h+1}=\rho_h+1\).  The two values of \(\pi\), and hence of
\(\pi^2\), are equal.  Equation \eqref{eq:E-id-pi2} gives
\(\chi_{h+1}-\chi_h=1\).

Otherwise \(\pi(\widehat m_{h+1})>\pi(\widehat m_h)\).  Since \(P\) begins
with \(00\), this case has \(h\geq2\).  First-passage minimality gives
\[
 h_A(\widehat m_h-1)=h-2,
 \qquad
 h_A(\widehat m_{h+1}-1)=h-1.
\]
The free-head contribution in \eqref{eq:four-head} increases by one.  The
closure \eqref{eq:theta-closure} shows that
\(\pi(\widehat m_h)\), \(\pi(\widehat m_{h+1})\), and their \(\pi\)-images
are record cuts of \(A\), or zero.  Their heights are nondecreasing and the
first pair increases by at least one.  Put
\[
 \epsilon_2=h_A(\pi(\widehat m_{h+1}))
             -h_A(\pi(\widehat m_h))\geq1,
 \qquad
 \epsilon_3=h_A(\pi^2(\widehat m_{h+1}))
             -h_A(\pi^2(\widehat m_h))\geq0.
\]
The four-head formula now gives
\[
 \chi_{h+1}-\chi_h=1+2\epsilon_2+\epsilon_3\geq3.
\]
The initial values are
\[
 \rho_1=1,
 \qquad \rho_2=2,
 \qquad \chi_1=1,
 \qquad \chi_2=2.
\]
Because \(P\) begins with \(00\), the second branch cannot occur at \(h=1\),
and these initial values give the first branch there.  At \(h=2\), the lower
argument may satisfy \(\pi(\widehat m_2)=0\).  The unified
formula \eqref{eq:four-head} covers precisely this boundary.  Equations
\eqref{eq:Adj} and \eqref{eq:Adj-gap} follow.
\end{proof}

\Needspace{5\baselineskip}
The three record statements hold along the whole orbit.
\begin{corollary}
\label{cor:orbital-record-transport}
For every \(r\geq0\), the transition \(P_r\to P_{r+1}\) satisfies RPP and
Adj, while \(P_r\) satisfies RC.
\end{corollary}

\begin{proof}
The core \(A_0=0101\) satisfies \(\mathsf{Cl}_\theta\) by direct evaluation
of its two zero cuts.  Corollary~\ref{cor:orbital-RS} supplies RS at every
level.  Theorem~\ref{thm:RPP-RC-Adj} gives RPP, RC of the next envelope, and
Adj.  Equivalence \eqref{eq:RC-Cl-equivalence} supplies
\(\mathsf{Cl}_\theta(A_{r+1})\), which closes the induction.
\end{proof}

\section{Fibers and amplitudes}
\label{sec:fibers}

Record transport now supplies the ordered fibers.

\subsection{Fibers and Toeplitz labels}
\label{sec:fiber-map}

For a nonempty Dyck word \(P\), define its first record cuts by
\[
 \rho_{P,h}=\min\{t:h_P(t)=h\}
 \qquad(0\leq h\leq V(P)).
\]
Every \(\rho_{P,h}\) belongs to the domain \(\mathcal D_\psi(P)\) of
\eqref{eq:compression-map}.  For \(h=0\), one has
\(\rho_{P,0}=0\) and \(O_P(0)=0<Z(P)\).  For \(h>0\),
\[
 O_P(\rho_{P,h})=Z_P(\rho_{P,h})-h
 \leq Z(P)-h<Z(P).
\]

\Needspace{5\baselineskip}
The record heights are organized by the following fiber map.
\begin{definition}
For \(0\leq h\leq V(P)\), put
\begin{equation*}
 g_P(h)=h_P\bigl(\psi_P(\rho_{P,h})\bigr).
\end{equation*}
Define
\[
 I(P)=\Img(g_P),
 \qquad
 \operatorname{Fib}_P(u)=\{h:g_P(h)=u\}.
\]
Since \(\psi_P(0)=p^P_1=1\), one has \(g_P(0)=1\), and hence \(1\in I(P)\).
For the orbital envelope \(P_j\), write
\[
 V_j=V(P_j),
 \qquad g_j=g_{P_j},
 \qquad I_j=I(P_j),
 \qquad \operatorname{Fib}_j(u)=\operatorname{Fib}_{P_j}(u).
\]
The components of \(I_j\) are its maximal intervals of consecutive integers,
and \(\Comp(I_j)\) denotes their set.
The component containing one is denoted by \(\mathcal S_L=[1,L]\) and is
termed special.  Every other component of length \(L\) is denoted by
\(\mathcal O_L\) and is termed ordinary.
\end{definition}

The RC property proved in Corollary~\ref{cor:orbital-record-transport} shows
that every value of \(g_j\) is a record height.  More generally, if
\[
 I(P)=\{u_0<u_1<\cdots<u_{m-1}\},
\]
put
\[
 f_P=(|\operatorname{Fib}_P(u_0)|,\ldots,
 |\operatorname{Fib}_P(u_{m-1})|).
\]
On the orbit, if
\[
 I_j=\{u_0<u_1<\cdots<u_{m_j-1}\},
\]
define the ordered fiber-size word
\begin{equation*}
 f_j=(f_{j,0},\ldots,f_{j,m_j-1}),
 \qquad f_{j,i}=|\operatorname{Fib}_j(u_i)|.
\end{equation*}

For \(L\geq1\) and \(1\leq s\leq L\), put
\begin{equation}
\label{eq:Pi}
 \Pi(L)_s=
 \begin{cases}
 2s+1,&s\leq\lfloor L/2\rfloor,\\
 2(L-s+1),&s>\lfloor L/2\rfloor.
 \end{cases}
\end{equation}
Thus \(\Pi(L)\) is a permutation of \(2,3,\ldots,L+1\).  Define
\begin{equation*}
 \Sigma(1)=(3),
 \qquad
 \Sigma(L)=(L+2,2L-2,2L-4,\ldots,2)\quad(L\geq2).
\end{equation*}

Let \(A\) be a nonempty anti-palindromic Dyck word of length \(2n\) satisfying
\(\mathsf S\).  Retain
its coordinates \(p_m,\omega_m,c_m,q_m,d_m\).  Put
\[
 \varrho=d_n,
 \qquad D=\max_m d_m,
 \qquad r_h=\min\{m:d_m=h\}.
\]
For \(h\geq1\), the first-passage vertex \(r_h\) is a leaf.  Edge-leaf
duality gives
\begin{equation*}
 \alpha_h=\omega_{n+1-r_h}.
\end{equation*}

Put \(B_T=T(A)\) and \(Q=T^2(A)\).  Let \(p'_k\) denote the zero cuts of
\(B_T\).  Use \(\pi,E\) from \eqref{eq:pi-E}.

\Needspace{5\baselineskip}
Two applications of $T$ act on the fiber map as follows.
\begin{lemma}
\label{lem:double-T-fibers}
Let \(A\) be a nonempty Dyck word of length \(2n\), put
\(B_T=T(A)\) and \(Q=T^2(A)\), and define \(\pi,E\) from \(A\).
For \(k=0\), or for
\(k\in\FP(E)\cap\{1,\ldots,2n+1\}\), the word \(Q\) has a positive record
of height \(E(k)+1\), and
\begin{equation*}
 g_Q(E(k)+1)=E(\pi(k))+1.
\end{equation*}
In addition, \(g_Q(0)=1\).
\end{lemma}

\begin{proof}
By Theorem~\ref{thm:T-records}, the first passages of \(e_{B_T}\) are zero
and the cuts \(p'_k\) with
\(k\in\FP(E)\cap\{1,\ldots,2n+1\}\).  The associated record of
\(Q=T(B_T)\) has height \(e_{B_T}(p'_k)+1=E(k)+1\).

The number of ones before that zero of \(Q\) is
\(\theta_{B_T}(p'_k)=p'_{\pi(k)}\).  Definition
\eqref{eq:compression-map} sends the record to the zero cut of rank
\(p'_{\pi(k)}+1\).  Its height is
\[
 p'_{\pi(k)}+1-\theta_{B_T}(p'_{\pi(k)})
 =p'_{\pi(k)}+1-p'_{\pi^2(k)}
 =E(\pi(k))+1.
\]
All terms are zero at \(k=0\), except the added one.  Finally,
\(\psi_Q(0)\) is the first zero cut of \(Q\), which has height one.
\end{proof}

Use the one cuts and sentinels of Definition~\ref{def:zero-coordinates} for
\(A\).
The fiber of \(\pi\) over \(p_r\), restricted to
\(\{1,\ldots,2n+1\}\), is
\begin{equation*}
 \mathcal K_r=\{s_r+1,\ldots,s_{r+1}\}.
\end{equation*}
Anti-palindromicity gives
\begin{equation*}
 |\mathcal K_r|=p_{n+1-r}-p_{n-r}
 \qquad(0\leq r\leq n).
\end{equation*}
In particular,
\begin{equation}
\label{eq:K-boundaries}
 |\mathcal K_0|=\varrho+1,
 \qquad
 |\mathcal K_{r_h}|=\omega_{n+1-r_h}=\alpha_h\quad(h\geq1).
\end{equation}

\Needspace{5\baselineskip}
The ordered fiber sizes read the Toeplitz profile.
\begin{theorem}
\label{thm:fiber-toeplitz}
Let \(A\) be a nonempty anti-palindromic Dyck word of length \(2n\) satisfying
\(\mathsf S\), let \(\varrho,D\), and
\((\alpha_1,\ldots,\alpha_D)\) be its root count, maximum depth, and Toeplitz
profile, and put \(Q=T^2(A)\).  Assume that \(A\) satisfies RS and that \(Q\)
satisfies RC.  The
ordered fiber-size word of \(Q\) is
\begin{equation*}
 f_Q=(\varrho+3,\alpha_1,\ldots,\alpha_D).
\end{equation*}
\end{theorem}

\begin{proof}
RS gives the disjoint partition
\begin{equation}
\label{eq:FP-partition}
 \FP(E)\cap\{1,\ldots,2n+1\}
 =\bigsqcup_{h=0}^{D}\mathcal K_{r_h}.
\end{equation}
By Lemma~\ref{lem:double-T-fibers}, all parameters in one set
\(\mathcal K_{r_h}\) have the same image under \(g_Q\).  Conversely, RS
shows that there are no other positive record parameters.  RC makes
\(\pi(k)=p_{r_h}\) another first-passage parameter or zero.  Since the cuts
\(p_{r_h}\) increase with \(h\), Lemma~\ref{lem:compressed-records}(ii) and
RC show that the values \(E(p_{r_h})+1\) increase strictly.  The fibers are therefore obtained in the order displayed in
\eqref{eq:FP-partition}.

For \(h\geq1\), equation \eqref{eq:K-boundaries} gives size \(\alpha_h\).
The first fiber contains the \(\varrho+1\) positive indices in
\(\mathcal K_0\).  It also contains the parameter \(k=0\) in Lemma
\ref{lem:double-T-fibers} and the domain value \(h=0\) of \(g_Q\).  Its size
is therefore \(\varrho+3\).  No other boundary term occurs.
\end{proof}

\subsection{Ordered fibers and components}

\Needspace{5\baselineskip}
The folded suspension factors are exactly the ordinary profiles.
\begin{lemma}
\label{lem:ordinary-fold}
For every \(L\geq1\),
\begin{equation*}
 \mathcal A_{L+1}=\Pi(L).
\end{equation*}
\end{lemma}

\begin{proof}
The \(s\)-th entry of \(\mathcal A_{L+1}\) is
\(\sigma_{L+1}(L+1-s)\).  The two branches of
\eqref{eq:sigma-gap} give exactly the two branches of \eqref{eq:Pi}.
The omitted last entry of \(\sigma_{L+1}\) is one, so the remaining entries
are a permutation of \(2,\ldots,L+1\).
\end{proof}

\Needspace{5\baselineskip}
The fiber words obey an ordered substitution.
\begin{theorem}
\label{thm:FW}
For every \(j\geq0\),
\begin{equation}
\label{eq:FW}
 f_{j+1}
 =\Sigma(f_{j,0}-1)\,
  \Pi(f_{j,1})\cdots\Pi(f_{j,m_j-1}).
\end{equation}
The juxtaposition is ordered concatenation.  The initial word is
\begin{equation}
\label{eq:f0}
 f_0=(3).
\end{equation}
\end{theorem}

\begin{proof}
The three record heights \(0,1,2\) of \(P_0=001011\) all map to height one,
which proves \eqref{eq:f0}.  The forest of \(A_0\) has root \((1)\) and
Toeplitz profile \((2)\).  Theorem~\ref{thm:fiber-toeplitz} gives
\(f_1=(4,2)=\Sigma(2)\).

For \(j\geq1\), Theorem~\ref{thm:fiber-toeplitz} writes \(f_{j+1}\) as its
first term followed by the Toeplitz profile of \(A_j\).  Apply
Theorem~\ref{thm:toeplitz-suspension} to \(A_{j-1}\).  Corollary
\ref{cor:orbital-roots} shows that the root prefix in the suspended profile
is
\[
 (2j+2,2j,\ldots,4,2).
\]
Write the Toeplitz profile of \(A_{j-1}\) as
\((\alpha^{(j-1)}_1,\ldots,\alpha^{(j-1)}_{D_{j-1}})\).  The preceding
fiber-Toeplitz correspondence gives
\[
 m_j=D_{j-1}+1,
 \qquad
 f_{j,h}=\alpha^{(j-1)}_h
 \quad(1\leq h\leq D_{j-1}).
\]
Since
\[
 f_{j,0}-1=j+2
 \qquad\text{and}\qquad
 f_{j+1,0}=j+4,
\]
the new first term followed by this prefix is
\(\Sigma(f_{j,0}-1)\).  Every remaining suspended factor is
\[
 \mathcal A_{\alpha^{(j-1)}_h+1}
 =\Pi(\alpha^{(j-1)}_h)=\Pi(f_{j,h})
\]
by Lemma~\ref{lem:ordinary-fold}.  The suspension orders these factors by
increasing \(h\), which proves \eqref{eq:FW}.
\end{proof}

For a finite subset \(I\subset\mathbb Z\), let
\(\operatorname{CompLengths}(I)\) be the ordered word of the lengths of its
maximal consecutive intervals, read from left to right.

\Needspace{5\baselineskip}
Adjacency matches the fiber factorization with the component structure.
\begin{lemma}
\label{lem:AdjFC}
Fix \(j\geq0\).  Assume RC for \(P_j\) and \(P_{j+1}\), together with RPP
and Adj for the transition \(P_j\to P_{j+1}\).  Then
\begin{equation}
\label{eq:FC}
 \operatorname{CompLengths}(I_{j+1})
 =(f_{j,0}-1,f_{j,1},\ldots,f_{j,m_j-1}).
\end{equation}
The order is the increasing order of the components of \(I_{j+1}\).
\end{lemma}

\begin{proof}
Write
\[
 \rho_h=\rho_{P_j,h},
 \qquad \mu_h=O_{P_j}(\rho_h)
 \qquad(0\leq h\leq V_j).
\]
By definition,
\[
 g_j(h)=h_{P_j}(p^{P_j}_{\mu_h+1}).
\]
The sequence \(\mu_h\) is nondecreasing.  RC makes every displayed cut a
strict record.  Hence
\begin{equation}
\label{eq:parent-fiber-adjacency}
 g_j(h+1)=g_j(h)
 \quad\Longleftrightarrow\quad
 \mu_{h+1}=\mu_h.
\end{equation}

RPP gives distinct one-count parameters \(\lambda_h\) of \(P_{j+1}\), in increasing order, for
the positive records \(h=1,\ldots,V_j\) of \(P_j\).  Put
\[
 \chi_h=h_{P_{j+1}}(p^{P_{j+1}}_{\lambda_h+1}).
\]
RC for \(P_{j+1}\) makes these strict record heights.  RPP is exhaustive, so
\begin{equation}
\label{eq:child-image}
 I_{j+1}=\{\chi_1<\chi_2<\cdots<\chi_{V_j}\}.
\end{equation}
Indeed, \(g_{j+1}(0)=1=\chi_1\), so the height-zero domain point adds no
value to this image.
Adj and \eqref{eq:parent-fiber-adjacency} give
\[
 \chi_{h+1}=\chi_h+1
 \quad\Longleftrightarrow\quad
 g_j(h+1)=g_j(h).
\]
Thus the maximal consecutive intervals in \eqref{eq:child-image} correspond,
in the same order, to the maximal intervals on which \(g_j\) is constant.
These are the ordered fibers of \(g_j\), after restriction to positive
heights.

Every \(P_j\) begins with \(00\).  Therefore \(\rho_0=0\), \(\rho_1=1\), and
\[
 g_j(0)=g_j(1)=1.
\]
Restricting to positive heights removes exactly \(h=0\) from the first fiber
and changes no other fiber.  This gives \eqref{eq:FC}.  The proof has not used
Theorem~\ref{thm:FW}, the multiplicity theorem below, or the binomial
amplitude.
\end{proof}

\subsection{ECO rules}

Ordered fibers give the ECO rules, which in turn give the binomial amplitude.

\Needspace{5\baselineskip}
Every component carries the same ordered multiplicities.
\begin{theorem}

For every \(j\geq0\) and every ordinary component
\(\mathcal O_L\subseteq I_j\), the ordered fiber sizes on that component are
\(\Pi(L)\).  In particular,
\begin{equation}
\label{eq:UM}
 \multiset{|\operatorname{Fib}_j(u)|:u\in\mathcal O_L}
 =\multiset{2,3,\ldots,L+1}.
\end{equation}
The special component of length \(L\) has ordered profile \(\Sigma(L)\).
\end{theorem}

\begin{proof}
The base image is \(I_0=\{1\}\), and
\(f_0=(3)=\Sigma(1)\).  Thus the assertion for the special component holds
at \(j=0\), while there is no ordinary component.

The factor lengths on the right side of \eqref{eq:FW} are
\[
 f_{j,0}-1,f_{j,1},\ldots,f_{j,m_j-1}.
\]
Lemma~\ref{lem:AdjFC} identifies them, in the same order, with the component
lengths of \(I_{j+1}\).  Therefore the factorization in
\eqref{eq:FW} is exactly the factorization of the ordered fiber word by its
components.  The first factor is \(\Sigma(L)\) and belongs to the unique
component containing one.  Each later factor of length \(L\) is \(\Pi(L)\)
and belongs to an ordinary component.  Lemma~\ref{lem:ordinary-fold} gives
\eqref{eq:UM}.  This proves both assertions by induction.
\end{proof}

The component succession rules are consequently
\begin{equation}
\label{eq:OECO}
 \mathcal O_L\longmapsto
 \mathcal O_2+\mathcal O_3+\cdots+\mathcal O_{L+1}
\end{equation}
and
\begin{equation}
\label{eq:SECO}
 \mathcal S_L\longmapsto
 \mathcal S_{L+1}+\mathcal O_2+\mathcal O_4+\cdots+\mathcal O_{2L-2}.
\end{equation}
To prove the rules, consider first an ordinary component \(\mathcal O_L\).
Its fiber profile \(\Pi(L)\) is a permutation of
\(2,3,\ldots,L+1\).  The component correspondence therefore creates one
ordinary component of each of these lengths, which is \eqref{eq:OECO}.

The profile of \(\mathcal S_L\) is \(\Sigma(L)\).  In the component
correspondence, one is subtracted from its first fiber size and from no other
fiber size.  The resulting ordered component lengths are
\[
 L+1,2L-2,2L-4,\ldots,2.
\]
The first is the new special component, and the remaining components are
ordinary.  This proves \eqref{eq:SECO}.  The ordinary sum is empty for
\(L=1\).  Since \(I_0\) consists of the single component \(\mathcal S_1\),
induction gives exactly one special component in every generation, with
\(I_{n+1}\) containing \(\mathcal S_{n+2}\).

\subsection{Generating functions}

At generation \(n\), consider the components of \(I_{n+1}\).  Its special
component is \(\mathcal S_{n+2}\).  Let \(o_{n,k}\) be the number of ordinary
components \(\mathcal O_k\), and put
\[
 \mathscr F_n(z)=\sum_{k\geq2}o_{n,k}z^{k-2},
 \qquad \mathscr F_0(z)=0.
\]
An ordinary component \(\mathcal O_k\) contributes
\[
 1+z+\cdots+z^{k-1}=\frac{1-z^k}{1-z}
\]
to the next ordinary-component polynomial.  The special component
\(\mathcal S_{n+2}\) contributes
\[
 1+z^2+\cdots+z^{2n}=\frac{1-z^{2n+2}}{1-z^2}.
\]
Summing these contributions gives
\begin{equation}
\label{eq:ECO-recurrence}
 \mathscr F_{n+1}(z)
 =\frac{\mathscr F_n(1)-z^2\mathscr F_n(z)}{1-z}
 +\frac{1-z^{2n+2}}{1-z^2}.
\end{equation}
Let
\[
 \mathscr F(t,z)=\sum_{n\geq0}\mathscr F_n(z)t^n,
 \qquad G(t)=\mathscr F(t,1).
\]
Summing \eqref{eq:ECO-recurrence} gives
\begin{equation}
\label{eq:ECO-kernel}
 (1-z+tz^2)\mathscr F(t,z)
 =tG(t)+\frac{t(1-z)}{(1-t)(1-tz^2)}.
\end{equation}

Let
\[
 C(t)=\frac{1-\sqrt{1-4t}}{2t},
 \qquad C(t)=1+tC(t)^2.
\]
This is the generating function of the Catalan numbers, OEIS A000108
\cite{OEIS}.
The substitution \(z=C(t)\) cancels the kernel.  The Catalan identity gives
\[
 1-C(t)=-tC(t)^2,
 \qquad
 1-tC(t)^2=2-C(t).
\]
Equation \eqref{eq:ECO-kernel} therefore yields
\[
 G(t)=\frac{tC(t)^2}{(1-t)(2-C(t))}.
\]
If
\[
 \kappa_n=\#\Comp(I_{n+1}),
 \qquad
 \mathscr K(t)=\sum_{n\geq0}\kappa_nt^n=\frac1{1-t}+G(t),
\]
then adding \(1/(1-t)\) gives
\begin{equation*}
 \mathscr K(t)=\frac1{(1-t)(2-C(t))}.
\end{equation*}

\Needspace{5\baselineskip}
The number of components equals the preceding arch height.
\begin{lemma}

For every \(r\geq0\),
\begin{equation}
\label{eq:a-V}
 \kappa_{r+1}=V_r.
\end{equation}
\end{lemma}

\begin{proof}
Lemma~\ref{lem:AdjFC} gives one component of \(I_{r+2}\) for each element of
\(I_{r+1}\), so
\[
 \#\Comp(I_{r+2})=|I_{r+1}|.
\]
For a record cut \(t\) of \(P_{r+1}\), the value of its fiber map depends only
on \(\lambda=O_{P_{r+1}}(t)\), because
\[
 \psi_{P_{r+1}}(t)=p^{P_{r+1}}_{\lambda+1}.
\]
RC makes the map
\(\lambda\mapsto h_{P_{r+1}}(p^{P_{r+1}}_{\lambda+1})\) strictly increasing on the
distinct one-count parameters supplied by RPP.  RPP is ordered and
exhaustive, and supplies exactly one such parameter for each of the \(V_r\)
positive record heights of \(P_r\).  The height-zero parameter gives
\(g_{r+1}(0)=g_{r+1}(1)=1\) and adds no image value.  Thus these parameters
are in bijection with all elements of \(I_{r+1}\), so
\(|I_{r+1}|=V_r\), which proves
\eqref{eq:a-V}.
\end{proof}

For \(r\geq1\),
\begin{equation}
\label{eq:W-a-difference}
 W_r=V_r-V_{r-1}=\kappa_{r+1}-\kappa_r.
\end{equation}
Extend the difference sequence at \(r=0\) by the artificial value
\(\kappa_1-\kappa_0=1\).  Its generating function is
\begin{align*}
\mathscr W(t)
&=\sum_{r\geq0}(\kappa_{r+1}-\kappa_r)t^r\\
&=\frac{(1-t)\mathscr K(t)-1}{t}
 =\frac{C(t)^2}{2-C(t)}
 =\frac{C(t)}{1-2tC(t)}.
\end{align*}
Equivalently,
\begin{equation}
\label{eq:marked-tree-equation}
 \mathscr W=C+2tC\mathscr W.
\end{equation}

\Needspace{5\baselineskip}
The amplitude increments are binomial.
\begin{theorem}
\label{thm:royal}
For every \(r\geq1\),
\begin{equation*}
 W_r=\binom{2r+1}{r}.
\end{equation*}
\end{theorem}

\begin{proof}
The Catalan series \(C\) counts full binary trees by their number of internal
vertices.  Equation \eqref{eq:marked-tree-equation} counts a full binary tree
with one arbitrary vertex marked.  The term \(C\) corresponds to a marked
root.  Otherwise, the mark lies in the left or right subtree, which gives
\(2tC\mathscr W\).  A full binary tree with \(r\) internal vertices has
\(2r+1\) vertices.  Hence
\[
 [t^r]\mathscr W
 =(2r+1)C_r
 =(2r+1)\frac1{r+1}\binom{2r}{r}
 =\binom{2r+1}{r}.
\]
For \(r\geq1\), this coefficient is \(W_r\) by
\eqref{eq:W-a-difference}.  The artificial coefficient at \(r=0\) is not an
amplitude assertion.
\end{proof}

\section{Asymptotics}
\label{sec:asymptotics}

The exact binomial amplitudes obtained in Section~\ref{sec:fibers} yield the
error estimates for \(\Qt(n)\).

\subsection{Rate of convergence}

Since \(V_0=2\), Theorem~\ref{thm:royal} gives
\begin{equation}
\label{eq:V-sum}
 V_r=2+\sum_{j=1}^r\binom{2j+1}{j}.
\end{equation}

\Needspace{5\baselineskip}
The amplitude grows at an explicit rate.
\begin{lemma}
\label{lem:V-asymptotic}
As \(r\to\infty\),
\begin{equation}
\label{eq:V-asymptotic}
 V_r\sim\frac{8}{3\sqrt\pi}\frac{4^r}{\sqrt r}.
\end{equation}
\end{lemma}

\begin{proof}
Put \(w_r=\binom{2r+1}{r}\).  Stirling's formula gives
\[
 w_r\sim\frac{2\,4^r}{\sqrt{\pi r}}.
\]
For every fixed \(j\geq0\),
\[
 \frac{w_{r-j}}{w_r}\longrightarrow4^{-j}.
\]
Indeed, for \(m\geq1\),
\[
 \frac{w_{m-1}}{w_m}=\frac{m+1}{2(2m+1)}\leq\frac13.
\]
Extend \(w_m\) by zero for \(m<0\).  Then
\(0\leq w_{r-j}/w_r\leq3^{-j}\) for every \(j\geq0\), so dominated
convergence in the backward sum, together with \(2/w_r\to0\), yields
\[
 \frac{V_r}{w_r}\longrightarrow\sum_{j\geq0}4^{-j}=\frac43.
\]
Multiplying the two asymptotics proves \eqref{eq:V-asymptotic}.
\end{proof}

Put
\begin{equation*}
 \varsigma(m)=\mathfrak a(m)+\mathfrak b(m)-m.
\end{equation*}
Equations \eqref{eq:ZP} and \eqref{eq:ZN} give
\begin{equation}
\label{eq:drift-values}
 \varsigma(m)=r+2\quad(u_r\leq m\leq v_r),
\end{equation}
and
\begin{equation*}
 \varsigma(m)\in\{r+2,r+3\}
 \quad(v_r\leq m\leq u_{r+1}).
\end{equation*}
The reconstructions are
\begin{equation}
\label{eq:Qt-even-reconstruction}
 \Qt(2m)=m-1+\varsigma(m)+\delta(m)
\end{equation}
and
\begin{equation}
\label{eq:Qt-odd-reconstruction}
 \Qt(2m-1)=m-1+\varsigma(m)-\delta(m).
\end{equation}

\Needspace{5\baselineskip}
The error term has optimal order.
\begin{theorem}
\label{thm:rate}
One has
\begin{equation}
\label{eq:rate}
 \Qt(n)=\frac n2+O\!\left(\frac{n}{\sqrt{\log n}}\right).
\end{equation}
The error is not \(o(n/\sqrt{\log n})\).  In particular,
\begin{equation}
\label{eq:limit-half}
 \lim_{n\to\infty}\frac{\Qt(n)}n=\frac12.
\end{equation}
\end{theorem}

The upper estimate is proved after the negative domination theorem below.
The lower-order assertion follows from the Gaussian subsequences in
Theorem~\ref{thm:layer-subsequences}.

\subsection{Positive and negative arches}

Let \(A\) be an orbital core of half-length \(n\).  Put
\[
 D=\max_{0\leq m\leq n}d_m,
 \qquad r_D=\min\{m:d_m=D\}.
\]
For the specific core \(A_r\), write
\[
 D_r=\max_{0\leq m\leq n_r}d_m(A_r),
 \qquad
 r_{D_r}=\min\{m:d_m(A_r)=D_r\}.
\]

\Needspace{5\baselineskip}
The first maximal depth sits at the center, and the center is preserved.
\begin{lemma}
\label{lem:first-peak-transport}
Assume that \(A\) satisfies RS and is anti-palindromic.  If \(r_D=n/2\), then
the core \(A'=\Lam(A)\), of half-length \(n'=4n+2\), satisfies
\[
 r'_{D'}=\frac{n'}2.
\]
\end{lemma}

\begin{proof}
Put \(B_T=T(A)\), and use \(\pi,E\) from \eqref{eq:pi-E}.  The gap formula
for \(\mathcal C(B_T)=A'\) gives
\[
 d_m^{A'}=e_{B_T}(m)
 \qquad(0\leq m\leq4n+2).
\]
Theorem~\ref{thm:T-records} gives
\[
 \FP(e_{B_T})=\{0\}\cup\{p'_k:k\in\FP(E)\},
 \qquad p'_k=k+\pi(k).
\]
RS says that the parameters producing first passages are exactly those for
which \(\pi(k)\) is a record of \(A\).

Use the one cuts \(s_j\) from Definition~\ref{def:zero-coordinates}.  The
largest parameter in the fiber above the
last record \(p_{r_D}\) is \(s_{r_D+1}\).  The first attainment of the maximal
value of \(e_{B_T}\) is therefore
\[
 r'_{D'}=p'_{s_{r_D+1}}=s_{r_D+1}+p_{r_D}.
\]
Anti-palindromicity gives
\[
 s_j=2n+1-p_{n+1-j}.
\]
Hence
\[
 r'_{D'}=2n+1+p_{r_D}-p_{n-r_D}.
\]
If \(r_D=n/2\), the last two cuts coincide, and the result is
\(2n+1=n'/2\).  All terminal parameters are included through the sentinels.
\end{proof}

\Needspace{5\baselineskip}
The first summit of every arch has an exact position.
\begin{theorem}

For every \(r\geq0\), the core \(A_r=P_r[1:|P_r|-1)\), whose half-length is
\(n_r=a_r-1\), satisfies
\begin{equation}
\label{eq:depth-bisection}
 2r_{D_r}=a_r-1.
\end{equation}
Consequently, with
\[
 \tau_r=\min\{t:h_{P_r}(t)=V_r\},
\]
one has
\begin{equation*}
 \tau_r=a_r+1-V_r.
\end{equation*}
\end{theorem}

\begin{proof}
For \(A_0=0101\), one has \(n_0=2\), \(D_0=1\), and \(r_{D_0}=1\).
Corollaries~\ref{cor:orbital-anti} and \ref{cor:orbital-RS}, followed by
Lemma~\ref{lem:first-peak-transport}, prove \eqref{eq:depth-bisection} by
induction.

At the first maximal depth,
\[
 p_{r_D}=r_D+q_{r_D}=2r_D-D=n_r-D.
\]
The envelope \(P_r=0A_r1\) adds one to the position and to the height.
Therefore \(V_r=D+1\) and
\[
 \tau_r=1+p_{r_D}=1+n_r-D=a_r+1-V_r.
\]
\end{proof}

The theorem concerns the first attainment only.  The maximum need not be
attained uniquely.  For example, \(P_0=001011\) attains \(V_0=2\) at cuts two
and four.

For the negative arch and \(1\leq t\leq|N_r|\), put
\[
 D_r^-(t)=h_{N_r}(t)+N_r[t-1].
\]
Equation \eqref{eq:negative-dictionary} gives
\(D_r^-(t)=-\delta(v_r+t)\).

For \(r\geq0\), put
\[
 M_r^+=\max_{0\leq t\leq2a_r}\frac{h_{P_r}(t)}{u_r+t},
 \qquad
 M_r^-=\max_{1\leq t\leq|N_r|}
 \frac{h_{N_r}(t)+N_r[t-1]}{v_r+t}.
\]

\Needspace{5\baselineskip}
The negative arches are dominated by the positive ones.
\begin{theorem}
\label{thm:negative-domination}
For every \(r\geq0\),
\begin{equation}
\label{eq:negative-domination}
 M_r^-\leq M_r^+.
\end{equation}
\end{theorem}

\begin{proof}
After \(s\) Law~1 outputs, let \(\xi_s,\upsilon_s\) be the numbers of consumed
letters.  Exact fit and Lemma~\ref{lem:law1-height} give
\[
 \xi_s+\upsilon_s=s
\]
and
\[
 h_{N_r}(s)=h_{P_r}(\xi_s+2)+h_{P_r}(\upsilon_s)-2.
\]
The definition of \(M_r^+\) gives
\[
 h_{P_r}(x)\leq M_r^+(u_r+x)
 \qquad(0\leq x\leq2a_r).
\]
Thus
\[
 h_{N_r}(s)\leq M_r^+(2u_r+s+2)-2.
\]
Since
\[
 D_r^-(t)=\max\{h_{N_r}(t-1),h_{N_r}(t)\},
\]
the affine upper bound just obtained increases with \(s\).  Its value at
\(s=t\) therefore dominates the bounds at both \(t-1\) and \(t\), and gives
\[
 D_r^-(t)\leq M_r^+(2u_r+t+2)-2.
\]
The identity \(2u_r+2=v_r-r\), together with \(M_r^+\geq0\), yields
\[
 D_r^-(t)\leq M_r^+(v_r+t).
\]
Division by \(v_r+t\) and maximization over the full negative domain prove
\eqref{eq:negative-domination}.
\end{proof}

\begin{proof}[Proof of the upper estimate in Theorem~\ref{thm:rate}]
On a positive arch,
\[
 \frac{|\delta(m)|}{m}\leq M_r^+\leq\frac{V_r}{u_r}.
\]
On the following negative arch, Theorem~\ref{thm:negative-domination} gives
the same upper bound.  By Lemma~\ref{lem:V-asymptotic} and
\(u_r\asymp4^r\), it is \(O(r^{-1/2})\).  Equations
\eqref{eq:drift-values} through \eqref{eq:Qt-odd-reconstruction} add only
\(O(r/m)\).  Since the level-\(r\) indices satisfy \(m\asymp4^r\) and
\(r\asymp\log m\), equation \eqref{eq:rate} follows.  It implies
\eqref{eq:limit-half}.
\end{proof}

\subsection{Orbital layers}

The root word of \(A_r\) is \(R_r=(1,3,\ldots,2r+1)\).  The root grammar
orders its depth layers as
\[
 c=\varphi^{2r+1}(R_r)\varphi^{2r}(R_r)\cdots\varphi(R_r)R_r.
\]
For \(0\leq k\leq2r+1\), define
\begin{equation}
\label{eq:B-layer}
 B_{r,k}=\sum_{i=0}^r\binom{2i+1}{k},
 \qquad
 T_{r,k}=\sum_{\ell=k}^{2r+1}B_{r,\ell}.
\end{equation}
At the upper boundary, set
\begin{equation}
\label{eq:layer-zero-boundary}
 B_{r,2r+2}=T_{r,2r+2}=0.
\end{equation}

\Needspace{5\baselineskip}
The forest layers give exact points on the positive arch.
\begin{lemma}
\label{lem:layer-points}
The number \(B_{r,k}\) is the width of layer \(k\) in the canonical forest.
Moreover,
\begin{equation*}
 d_{T_{r,k}}=B_{r,k}.
\end{equation*}
The positive arch contains the exact point
\begin{equation}
\label{eq:layer-point}
 t_{r,k}=1+2T_{r,k}-B_{r,k},
 \qquad
 h_{P_r}(t_{r,k})=B_{r,k}+1.
\end{equation}
\end{lemma}

\begin{proof}
Equation \eqref{eq:phi-binomial-length} gives the layer width
\(B_{r,k}\).  After the layers \(2r+1,\ldots,k\) have been read, the number
of degrees consumed is \(T_{r,k}\).  The sum of the degrees in layer
\(\varphi^\ell(R_r)\) equals the length of the next layer.  With the
boundary convention \eqref{eq:layer-zero-boundary},
\[
 q_{T_{r,k}}
 =\sum_{\ell=k}^{2r+1}B_{r,\ell+1}
 =T_{r,k+1}.
\]
Therefore
\[
 d_{T_{r,k}}=T_{r,k}-T_{r,k+1}=B_{r,k}.
\]
A cut after the \(m\)-th zero of a core occurs at
\(p_m=2m-d_m\).  The envelope adds one to its cut and height.  Substitution
of \(m=T_{r,k}\) proves \eqref{eq:layer-point}.
\end{proof}

The core half-length satisfies
\begin{equation}
\label{eq:n-layer-sum}
 n_r=a_r-1
 =\sum_{k=0}^{2r+1}B_{r,k}
 =\sum_{i=0}^r2^{2i+1},
 \qquad u_r=2n_r-r.
\end{equation}

Let \(\Phi\) denote the distribution function of the standard normal law.
For \(z>0\), put
\begin{equation*}
 F(z)=\frac{3e^{-z^2/2}}{2(1+\Phi(-z))}.
\end{equation*}

For fixed \(z>0\) and all sufficiently large \(r\), put
\begin{equation*}
 k_r(z)=\left\lfloor
 \frac{2r+1}{2}+z\frac{\sqrt{2r+1}}2
 \right\rfloor.
\end{equation*}

\Needspace{5\baselineskip}
The layer widths and the tails have Gaussian limits.
\begin{lemma}
\label{lem:layer-CLT}
For every fixed \(z>0\),
\begin{equation}
\label{eq:layer-CLT}
 \frac{B_{r,k_r(z)}}{V_r}\longrightarrow e^{-z^2/2},
 \qquad
 \frac{T_{r,k_r(z)}}{n_r}\longrightarrow\Phi(-z).
\end{equation}
\end{lemma}

\begin{proof}
Write \(i=r-j\).  For every fixed \(j\), the standardized displacement of
\(k_r(z)\) from the mean of
\(\operatorname{Bin}(2(r-j)+1,1/2)\) tends to \(z\).  The local central
limit theorem therefore gives
\[
 \binom{2(r-j)+1}{k_r(z)}
 \sim\frac{2\,4^{r-j}}{\sqrt{\pi r}}e^{-z^2/2}.
\]
The standard bound, with an absolute constant \(\mathcal L\),
\[
 \max_k\binom Nk\leq \mathcal L\frac{2^N}{\sqrt N}
\]
gives, after division by \(4^r/\sqrt r\), a summable geometric majorant
\(\mathcal L'4^{-j}\) for \(j\leq r/2\), since
\(r/(r-j)\leq2\).  For fixed \(z>0\) and sufficiently large \(r\), the
terms with \(j>r/2\) vanish because then
\(2(r-j)+1<k_r(z)\).  Dominated convergence in \eqref{eq:B-layer} yields
\[
 B_{r,k_r(z)}
 \sim\frac{8}{3\sqrt\pi}\frac{4^r}{\sqrt r}e^{-z^2/2}.
\]
Lemma~\ref{lem:V-asymptotic} proves the first limit.

For the tail, let \(X_i\sim\operatorname{Bin}(2i+1,1/2)\).  Then
\[
 T_{r,k_r(z)}
 =\sum_{i=0}^r2^{2i+1}\Pr\{X_i\geq k_r(z)\}.
\]
For fixed \(j\), the central limit theorem gives a limiting probability
\(\Phi(-z)\).  The normalized weights satisfy
\[
 \frac{2^{2(r-j)+1}}{n_r}\longrightarrow\frac34\,4^{-j}
\]
 and, for every \(r,j\), are at most \(4^{-j}\), because the top summand
\(2^{2r+1}\) is part of \(n_r\).  After extending the summands by zero for
\(j>r\), a second dominated convergence proves the tail limit.
\end{proof}

\Needspace{5\baselineskip}
A binomial layer asymptotically exceeds the first-summit quotient.
\begin{proposition}
\label{prop:layer-enhancement}
For every fixed \(z>0\),
\begin{equation}
\label{eq:M-layer-lower}
 \liminf_{r\to\infty}
 \frac{M_r^+}{V_r/(3a_r)}\geq F(z).
\end{equation}
Moreover, \(F(1/4)>1\).
\end{proposition}

\begin{proof}
Equations \eqref{eq:layer-point} and \eqref{eq:n-layer-sum} give
\[
 M_r^+\geq
 \frac{B_{r,k}+1}
 {2n_r-r+1+2T_{r,k}-B_{r,k}}.
\]
Choose \(k=k_r(z)\).  The layer width is \(O(V_r)=o(a_r)\), while
\(n_r/a_r\to1\).  Lemma~\ref{lem:layer-CLT} and
\eqref{eq:layer-point} give explicitly
\[
 \frac{B_{r,k_r(z)}}{V_r}\longrightarrow e^{-z^2/2},
 \qquad
 \frac{u_r+t_{r,k_r(z)}}{a_r}
 \longrightarrow2(1+\Phi(-z)).
\]
Dividing the lower bound by \(V_r/(3a_r)\) therefore gives \(F(z)\), which
proves \eqref{eq:M-layer-lower}.

For the strict inequality, put
\[
 \phi(s)=\frac{e^{-s^2/2}}{\sqrt{2\pi}},
 \qquad \iota=\int_0^{1/4}\phi(s)\,ds.
\]
Since \(\Phi(-1/4)=1/2-\iota\), the inequality \(F(1/4)>1\) is equivalent to
\[
 1-e^{-1/32}<\frac23\iota.
\]
The left side is less than \(1/32\).  Since \(\phi\) decreases on
\([0,1/4]\),
\[
 \frac23\iota>
 \frac{e^{-1/32}}{6\sqrt{2\pi}}
 >\frac{31/32}{18}
 =\frac{31}{576}
 >\frac1{32}.
\]
Only \(e^{-x}>1-x\) and \(\sqrt{2\pi}<3\) were used.
\end{proof}

\subsection{Limsup bounds}

For fixed \(z>0\), define
\begin{equation}
\label{eq:x-r-z}
 x_r(z)=u_r+t_{r,k_r(z)},
 \qquad \widehat n_r(z)=2x_r(z).
\end{equation}

\Needspace{5\baselineskip}
Explicit subsequences realize each Gaussian level.
\begin{theorem}
\label{thm:layer-subsequences}
For every fixed \(z>0\),
\begin{equation*}
 \lim_{r\to\infty}
 \left(\frac{\Qt(\widehat n_r(z))}{\widehat n_r(z)}-\frac12\right)
 \sqrt{\log_2 \widehat n_r(z)}
 =\frac{F(z)}{3\sqrt{2\pi}}.
\end{equation*}
\end{theorem}

\begin{proof}
At the point \eqref{eq:x-r-z}, Lemma~\ref{lem:layer-points} and
\eqref{eq:drift-values} give
\[
 \delta(x_r(z))=B_{r,k_r(z)}+1,
 \qquad \varsigma(x_r(z))=r+2.
\]
The even reconstruction is therefore the exact identity
\[
 \frac{\Qt(2x_r(z))}{2x_r(z)}-\frac12
 =\frac{B_{r,k_r(z)}+r+2}{2x_r(z)}.
\]
The layer width is comparable with \(V_r\), while \(r=o(V_r)\).  Equations
\eqref{eq:V-asymptotic}, \eqref{eq:layer-CLT}, and
\eqref{eq:n-layer-sum} give
\[
 \frac{x_r(z)}{a_r}\longrightarrow2(1+\Phi(-z)),
 \qquad
 \frac{B_{r,k_r(z)}}{V_r}\longrightarrow e^{-z^2/2},
\]
and hence
\[
 \frac{B_{r,k_r(z)}+r+2}{2x_r(z)}
 \sim\frac{F(z)V_r}{6a_r}.
\]
Also \(x_r(z)\asymp a_r\) and
\[
 \log_2(2x_r(z))=2r+O(1).
\]
Since \(V_r/a_r\sim1/\sqrt{\pi r}\), the limit is
\(F(z)/(3\sqrt{2\pi})\).
\end{proof}

\Needspace{5\baselineskip}
The normalized limsup is bounded on both sides.
\begin{theorem}
\label{thm:limsup-bounds}
Put
\begin{equation}
\label{eq:layer-constant}
 C_{\mathrm{layer}}
 =\frac1{3\sqrt{2\pi}}\sup_{z>0}F(z).
\end{equation}
Then
\begin{equation}
\label{eq:limsup-chain}
 C_{\mathrm{layer}}
 \leq\limsup_{n\to\infty}
 \left(\frac{\Qt(n)}n-\frac12\right)\sqrt{\log_2 n}
 \leq\limsup_{n\to\infty}
 \left|\frac{\Qt(n)}n-\frac12\right|\sqrt{\log_2 n}
 \leq\frac1{2\sqrt{2\pi}}.
\end{equation}
In particular,
\begin{equation}
\label{eq:layer-strict-lower}
 \limsup_{n\to\infty}
 \left|\frac{\Qt(n)}n-\frac12\right|\sqrt{\log_2 n}
 \geq\frac{F(1/4)}{3\sqrt{2\pi}}
 >\frac1{3\sqrt{2\pi}}.
\end{equation}
\end{theorem}

\begin{proof}
Theorem~\ref{thm:layer-subsequences} holds for every fixed \(z>0\).  Taking
the supremum of these lower bounds proves the first inequality in
\eqref{eq:limsup-chain}.  The second follows from \(x\leq |x|\).  Theorem
\ref{thm:layer-subsequences} at \(z=1/4\), together with the exact inequality
\(F(1/4)>1\) in Proposition~\ref{prop:layer-enhancement}, gives
\eqref{eq:layer-strict-lower}.

For the upper bound, define the complete blocks
\[
 J_r=\{2u_r-1,\ldots,2u_{r+1}-2\}.
\]
They partition a tail of the positive integers.  Each \(n\in J_r\) is
uniquely \(2m\) or \(2m-1\), with \(m\in[u_r,u_{r+1})\).  Uniformly on \(J_r\),
\[
 \log_2 n=2r+O(1).
\]
Equations \eqref{eq:Qt-even-reconstruction} and
\eqref{eq:Qt-odd-reconstruction} give
\begin{equation*}
 \left|\frac{\Qt(n)}n-\frac12\right|
 =\frac{|\delta(m)|}{2m}+O\!\left(\frac r m\right)
\end{equation*}
uniformly for both parities.  Therefore
\[
 \max_{n\in J_r}
 \left|\frac{\Qt(n)}n-\frac12\right|
 =\frac12\max(M_r^+,M_r^-)+o(r^{-1/2}).
\]
The definition of \(M_r^+\) includes the shared cut \(v_r\), whereas that of
\(M_r^-\) includes the additional cut \(u_{r+1}\).  The corresponding depths are zero, so
these endpoints do not change the maximum over the half-open block.
Theorem~\ref{thm:negative-domination} and the elementary bound
\(M_r^+\leq V_r/u_r\) give
\[
 \max_{n\in J_r}
 \left|\frac{\Qt(n)}n-\frac12\right|
 \leq\frac{V_r}{2u_r}+o(r^{-1/2}).
\]
Now
\[
 u_r\sim\frac{16}{3}4^r,
 \qquad
 V_r\sim\frac{8}{3\sqrt\pi}\frac{4^r}{\sqrt r}.
\]
Multiplication by \(\sqrt{2r}\) gives
\[
 \frac12\cdot\frac1{2\sqrt{\pi r}}\cdot\sqrt{2r}
 =\frac1{2\sqrt{2\pi}}.
\]
Since the blocks cover a tail, this is the last inequality in
\eqref{eq:limsup-chain}.
\end{proof}

\begin{proof}[Proof of the little-oh assertion in Theorem~\ref{thm:rate}]
Theorem~\ref{thm:layer-subsequences} gives an even subsequence on which
\[
 \left|\frac{\Qt(n)}n-\frac12\right|\sqrt{\log n}
\]
has a positive limit, because \(\log n=(\log 2)\log_2 n\).  Thus the error
cannot be \(o(n/\sqrt{\log n})\).
\end{proof}

\begin{proof}[Proof of the main theorem]
Theorem~\ref{thm:well-definedness} proves item (i), with its limit supplied by
Theorem~\ref{thm:rate}.  Theorem~\ref{thm:royal}, equation
\eqref{eq:V-sum}, and Lemma~\ref{lem:V-asymptotic} prove item (ii).
The two parts of Theorem~\ref{thm:rate} prove item (iii).
Theorems~\ref{thm:layer-subsequences} and \ref{thm:limsup-bounds} prove item
(iv).
\end{proof}

Numerical maximization of the explicit function \(F\) gives a value near
\(1.03785617\) at
\(z\) near \(0.27603\), and hence
\[
 C_{\mathrm{layer}}\approx0.1380149024.
\]
For orientation,
\[
 \frac{F(1/4)}{3\sqrt{2\pi}}\approx0.1379682659,
 \qquad
 \frac1{2\sqrt{2\pi}}\approx0.1994711402.
\]
The exact definition of \(C_{\mathrm{layer}}\) is the supremum in
Theorem~\ref{thm:limsup-bounds}.

It remains open whether
\[
 \limsup_{n\to\infty}
 \left(\frac{\Qt(n)}n-\frac12\right)\sqrt{\log_2 n}
 =C_{\mathrm{layer}}
\]
holds.  No upper bound matching the binomial-layer lower bound is known.  The
proved bound \(1/(2\sqrt{2\pi})\) is uniform but does not match this layer
family asymptotically.

\appendix
\numberwithin{equation}{section}

\section{Breadth-first separators}
\label{app:separators}

This appendix proves Lemma~\ref{lem:separators} with all order and boundary
conventions visible.

In the original depth order, put
\[
 \Xi_{m,t}=\mathbf1_{\{q_{m-1}<t\leq q_m\}}.
\]
Differencing \eqref{eq:incidence-centro} gives
\[
 \Xi_{m,t}=1
 \quad\Longleftrightarrow\quad
 q_{\mathcal R(t)}=\mathcal R(m).
\]
After reversing the vertices, the parent of \(v\) is therefore \(q_v\), with
\(q_v=0\) meaning that \(v\) is a root.  Since \(q_v<v\) and \(q_v\) is
nondecreasing, roots occur first and the children of every parent form a
contiguous interval.  Induction on \(v\) shows that depths are nondecreasing.
Thus \(1,\ldots,n\) is exactly the old breadth-first order.  The number of
children of \(u\) is
\[
 \#\{v:q_v=u\}
 =q_{\mathcal R(u)}-q_{\mathcal R(u)-1}
 =c_{\mathcal R(u)}=\check c_u.
\]

After suspension, denote the separate degree-one root by \(Y_0\).  Put
\[
 Y_u=y_u\quad(u\geq1),
 \qquad Z_1=\text{the child of }Y_0,
 \qquad Z_{u+1}=\zeta_u\quad(u\geq1).
\]
The new roots are
\[
 \operatorname{core}(1),\ldots,\operatorname{core}(\varrho),Y_0.
\]
If \(v_1<\cdots<v_{\check c_u}\) are the old children of \(u\), the ordered child
factor of \(\operatorname{core}(u)\) is
\[
 x_u,\operatorname{core}(v_1),\ldots,\operatorname{core}(v_{\check c_u}),Y_u.
\]
The only child of \(Y_u\) is \(Z_{u+1}\), including \(u=0\).

Let \(\prec\) denote the new breadth-first order and \(\operatorname{rk}\)
its rank.  With \(N_0=\varrho+1\) roots, the first child position of a vertex
\(V\) is
\[
 N_0+1+\sum_{X\prec V}\deg X.
\]
Consequently,
\begin{align*}
 \operatorname{rk}(x_u)
 &=N_0+1+\sum_{X\prec \operatorname{core}(u)}\deg X,\\
 \operatorname{rk}(\operatorname{core}(v_j))
 &=\operatorname{rk}(x_u)+j,\\
 \operatorname{rk}(Y_u)
 &=\operatorname{rk}(x_u)+\check c_u+1,\\
 \operatorname{rk}(Z_{u+1})
 &=N_0+1+\sum_{X\prec Y_u}\deg X.
\end{align*}

Fix \(1\leq u<n\), and put
\[
 p=q_u,
 \qquad p'=q_{u+1}.
\]
The noncore vertices strictly between \(\operatorname{core}(u)\) and \(\operatorname{core}(u+1)\) are
\begin{equation}
\label{eq:x-separators}
 \{x_v:p<v\leq p'\},
\end{equation}
\begin{equation}
\label{eq:y-separators}
 \{Y_v:p\leq v<p'\},
\end{equation}
and
\begin{equation}
\label{eq:z-separators}
 \{Z_w:q_p<w\leq q_{p'}\}.
\end{equation}
These are equalities of ordered sets.  We prove them by strong induction on
\(u\).  The inequalities \(p'\leq u\) ensure that every core boundary used
below has smaller index.

If \(p=p'=0\), the two cores are consecutive roots.  If \(0<p=p'\), they are
consecutive core children of one parent.  All three intervals are empty.

Suppose \(p=0<p'\).  Then \(u=\varrho\), the core \(\operatorname{core}(u+1)\) is the first
nonroot core, \(p'\) is an old root, and \(q_{p'}=0\).  After
\(\operatorname{core}(\varrho)\) comes \(Y_0\).  In the next layer, the root child factors
through \(p'-1\) contain \(x_v,Y_v\), and the factor of \(\operatorname{core}(p')\) begins
with \(x_{p'}\).  The vertices between the two cores are therefore
\[
 x_1,\ldots,x_{p'},
 \qquad
 Y_0,\ldots,Y_{p'-1}.
\]
There is no \(Z\)-separator because \(q_{p'}=q_0=0\).  The vertex \(Z_1\),
which is the child of \(Y_0\), appears after \(\operatorname{core}(u+1)\).  This is the root
sentinel case.

Suppose \(0<p<p'\).  The core \(\operatorname{core}(u)\) is the last core child of \(\operatorname{core}(p)\),
and \(\operatorname{core}(u+1)\) is the first core child of \(\operatorname{core}(p')\).  A parent \(v\) with
\(p<v<p'\) has no old child, because \(q\) is nondecreasing.  The end of the
factor of \(\operatorname{core}(p)\) contributes \(Y_p\).  Each intervening childless parent
contributes \(x_v,Y_v\).  The factor of \(\operatorname{core}(p')\) contributes \(x_{p'}\)
before its first core child.  This proves \eqref{eq:x-separators} and
\eqref{eq:y-separators}.

For the \(Z\)-vertices, concatenate the already proved boundaries between
\(\operatorname{core}(p)\) and \(\operatorname{core}(p')\).  The \(Y\)-indices encountered are exactly
\[
 q_p,q_p+1,\ldots,q_{p'}-1.
\]
Their children are \(Z_{q_p+1},\ldots,Z_{q_{p'}}\), which proves
\eqref{eq:z-separators}.  If \(q_p=0<q_{p'}\), the first such vertex is
\(Y_0\), and its child is exactly \(Z_1\).  Breadth-first concatenation
preserves the increasing index order in all three families.

The first-child rank formula gives
\[
 K_u=N_0+\sum_{X\prec \operatorname{core}(u)}(\deg X-1).
\]
Hence
\[
 K_{u+1}-K_u
 =\sum_{\operatorname{core}(u)\preceq X\prec \operatorname{core}(u+1)}(\deg X-1).
\]
The core \(\operatorname{core}(u)\) contributes \(\check c_u+1\).  Each \(Y\)-separator contributes
zero.  Each \(x\)- or \(Z\)-separator contributes \(-1\).  Put
\(L=c_{u+1}=p'-p\).  There are \(L\) vertices of type \(x\), and the number
of type \(Z\) is
\[
 q_{p'}-q_p=\sum_{v=p+1}^{p'}c_v.
\]
Property \(\mathsf S\) at \(u+1\) says that these degrees are the permutation
\(\varphi(L)\) of \(0,\ldots,L-1\).  Therefore
\[
 q_{p'}-q_p=\binom L2.
\]
It follows that
\[
 K_{u+1}-K_u
 =\check c_u+1-L-\binom L2
 =\check c_u+1-\binom{L+1}{2},
\]
which is \eqref{eq:separator-step}.

Finally, \(\operatorname{core}(1)\) is the first root and \(x_1\) is the first nonroot.  Their
ranks are one and \(\varrho+2\), so \(K_1=\varrho+1\).  If \(L=0\), both
separator intervals are empty and the same formula gives
\(K_{u+1}-K_u=\check c_u+1\).  If \(n=1\), only \(K_1\) occurs.  The step formula is
quantified only for \(u<n\), so no symbol \(q_{n+1}\) is introduced.
The virtual parent \(n+1\) of the depth-first description becomes parent zero
of the roots in breadth-first order.  The vertices \(Y_0\) and \(Z_1\)
realize exactly this sentinel, including the boundary between roots and the
first nonroot core.

\section{Toeplitz bases}
\label{app:toeplitz-bases}

We verify \eqref{eq:base-toeplitz-profiles} by listing all active edges.

For \(R=0^b\), the degree word is \(c=0^b\), and \(D=b\).  There are no real
active edges.  The virtual edges \(t\to b+1\), \(1\leq t\leq b\), have spans
\(b+1-t\) and labels \(\omega_t=1\).  They cover every span from one through \(b\)
once, giving profile \(1^b\).

For \(R=0^a1^b\),
\[
 c=0^{a+b}1^b,
 \qquad n=a+2b,
 \qquad q_n=b,
 \qquad D=a+b.
\]
Vertices of degree one have only a terminal child, so no real edge is active.
For the \(a\) degree-zero roots, the exact indices, labels, and spans are
\[
 t=b+i,
 \qquad \omega_t=1,
 \qquad h=a+b+1-i
 \qquad(1\leq i\leq a).
\]
For the \(b\) degree-one roots they are
\[
 t=a+b+j,
 \qquad \omega_t=2,
 \qquad h=b+1-j
 \qquad(1\leq j\leq b).
\]
Thus the first family has spans \(b+1,\ldots,a+b\), and the second has spans
\(1,\ldots,b\).  The
profile is \(2^b1^a\).  The first interval is empty when \(a=0\).

For \(R=1^a2^b\),
\[
 c=0^{a+b}(1,0)^b1^a2^b,
 \qquad n=2a+4b,
 \qquad q_n=a+3b,
 \qquad D=a+2b.
\]
The virtual edges of the degree-one roots have exact data
\[
 t=a+3b+i,
 \qquad \omega_t=2,
 \qquad h=a+b+1-i
 \qquad(1\leq i\leq a),
\]
and those of the degree-two roots have
\[
 t=2a+3b+j,
 \qquad \omega_t=3,
 \qquad h=b+1-j
 \qquad(1\leq j\leq b).
\]

The only real active edges are the first children of the \(b\) degree-two
vertices.  For
\[
 m=2a+3b+j\qquad(1\leq j\leq b),
\]
one has
\[
 q_{m-1}=a+b+2j-2,
 \qquad
 t=q_{m-1}+1=a+b+2j-1.
\]
It has label two and span \(a+2b-j+1\).  These spans are exactly
\(a+b+1,\ldots,a+2b\).  The second child is terminal.  Thus the complete
profile is \(3^b2^{a+b}\).

\section{Boundary cases}

This appendix collects the boundary identities that connect the different
coordinate systems.

\begin{enumerate}[label=\textup{(\roman*)}]
\item The arch words use half-open increment intervals
\([u_r,v_r)\) and \([v_r,u_{r+1})\).  Their height profiles include both
terminal cuts.  Equation \eqref{eq:negative-dictionary} is quantified by
\(1\leq t\leq|N_r|\), so the letter \(N_r[t-1]\) is always defined.

\item The negative boundary values are
\[
 D_r^-(1)=1,
 \qquad D_r^-(|N_r|)=0.
\]
The first follows from \(N_r[0]=0\).  The second follows from the terminal
letters \(11\) and the last two heights zero and \(-1\).

\item Definition~\ref{def:zero-coordinates} fixes
\(p_0=q_0=d_0=0\) and the terminal sentinels
\(p_{n+1}=s_{n+1}=2n+1\).  The latter include the last point
\(k=2n+1\) in the terminal ramp \(\mathcal K_n\), but do not define an extra
gap.

\item The reflection is extended by
\[
 \mathcal R(0)=n+1,
 \qquad \mathcal R(n+1)=0,
 \qquad \mathcal M_{0,t}=\mathcal M_{i,n+1}=0.
\]
Virtual edges are exactly \(t\to n+1\) with \(q_n<t\leq n\).  A terminal
child \(t=q_m\) is not active.

\item For an orbital transition, the one-count parameter zero is obtained
from the first \(k\) with \(\pi(k)=0\).  The one-count parameter one is obtained from the
nonempty fiber with \(\pi(k)=p_1=1\).  RS makes both parameters first
passages of \(E\).  Since \(\pi(1)=0\), their values
\(\pi(k)+\pi^2(k)\) are zero and one.  They produce the RPP padding record
cuts one and two.

\item The first fiber in Theorem~\ref{thm:fiber-toeplitz} has
\(\varrho+1\) positive parameters in \(\mathcal K_0\).  The parameter
\(k=0\) and the domain height \(h=0\) add exactly two further elements.
No later fiber receives either addition.

\item In Lemma~\ref{lem:AdjFC}, \(P_j\) begins with \(00\), so
\(g_j(0)=g_j(1)=1\).  Passing from all heights to positive heights removes
only \(h=0\) from the first fiber.  This is the unique subtraction in
\eqref{eq:FC}.

\item The smallest folded profiles are
\[
 \Pi(1)=(2),
 \qquad \Sigma(1)=(3),
 \qquad f_0=(3).
\]
The ordinary multiplicity statement at level zero is vacuous.
\end{enumerate}

\section*{Funding}

This research did not receive any specific grant from funding agencies in the
public, commercial, or not-for-profit sectors.

\section*{Declaration of competing interest}

The author declares that he has no known competing financial interests or
personal relationships that could have appeared to influence the work
reported in this paper.

\section*{Declaration of generative AI and AI-assisted technologies in the
manuscript preparation process}

During the preparation of this work, the author used OpenAI ChatGPT 5.5 and 5.6 to assist with language editing, manuscript organization, and exploratory computations. The author reviewed and edited the resulting material and takes full responsibility for the content of the article.

\end{document}